\documentclass[11pt]{amsart}
\usepackage{latexsym,xspace,enumerate,amsfonts,amsmath,amssymb,amscd}
\usepackage{xypic,color,hyperref,syntonly,slashed,paralist,stmaryrd,graphicx}
\usepackage[arrow, matrix, curve]{xy} 
\usepackage{bbm}

\newcommand{\e}{\epsilon}
\newcommand{\app}{\mathrm{app}}

\newcommand{\D}{\mathbb{D}}
\newcommand{\N}{\mathbb{N}}
\newcommand{\Z}{\mathbb{Z}}

\newcommand{\R}{\mathbb{R}}
\newcommand{\C}{\mathbb{C}}

\newcommand{\SL}{\mr{SL}}

\newcommand{\U}{\mr{U}}
\newcommand{\SU}{\mr{SU}}

\newcommand{\mf}{\mathfrak}
\newcommand{\mr}{\mathrm}

\newcommand{\mc}{\mathcal}

\newcommand{\End}{\mathop{\rm End}\nolimits}
\newcommand{\Hom}{\mathop{\rm Hom}\nolimits}

\renewcommand{\ker}{\mathop{\rm ker}\nolimits}

\renewcommand{\deg}{\mathop{\rm deg}\nolimits}

\newcommand{\Tr}{\mathop{\rm Tr}\nolimits}

\renewcommand{\mod}{\mathop{\rm mod}\nolimits}

\newcommand{\delb}{\bar\partial}

\newcommand{\note}[1]{\marginpar{\raggedright\if@twoside\ifodd\c@page\raggedleft\fi\fi\sf\scriptsize \red{RMK: #1}}}

\newcommand\red[1]{\textcolor{red}{#1}}

\newcommand{\be}{\begin{equation}}
\newcommand{\ben}{\begin{equation}\nonumber}
\newcommand{\ee}{\end{equation}}
\newcommand{\bp}{\begin{para}}
\newcommand{\ep}{\end{para}}
\newcommand{\bps}{\begin{paras}}
\newcommand{\eps}{\end{paras}}

\newcommand{\su}{\mathfrak{su}}

\def\pr{\textrm{pr}}

\newtheorem{proposition}{\textbf{Proposition}}
\newtheorem{lemma}[proposition]{\textbf{Lemma}}
\newtheorem{corollary}[proposition]{\textbf{Corollary}}
\newtheorem{theorem}[proposition]{\textbf{Theorem}}

\theoremstyle{definition}

\newtheorem*{example*}{\textbf{Example}}

\newtheorem*{theorem*}{\textbf{Theorem}}

\newtheorem{step}{Step}

\newcounter{para}[section]
\newenvironment{para}[2][]{\refstepcounter{para}\noindent\ignorespaces{\bf #1\thepara. #2.} \rmfamily}{\noindent\ignorespacesafterend\bigskip}
\newenvironment{paras}[1]{\noindent\ignorespaces{\bf #1.} \rmfamily}{\noindent\ignorespacesafterend\bigskip}

\numberwithin{proposition}{section}
\numberwithin{definition}{section}

\begin{document}
\title[Higgs bundles on degenerating  Riemann surfaces]{Moduli spaces of Higgs bundles on degenerating   Riemann surfaces}
\date{\today}

\author{Jan Swoboda}
\address{Mathematisches Institut der LMU M\"unchen\\Theresienstra{\ss}e 39\\D--80333 M\"unchen\\ Germany}
\email{swoboda@math.lmu.de}


\begin{abstract}
We prove a gluing theorem for   solutions $(A_0,\Phi_0)$ of Hitchin's self-duality equations with logarithmic singularities    on a  rank-$2$ vector bundle over a  noded Riemann surface $\Sigma_0$ representing a boundary point of Teichm\"uller moduli space. We show that  every nearby smooth Riemann surface $\Sigma_1$  carries a smooth solution $(A_1,\Phi_1)$ of the self-duality equations, which may be viewed as a    desingularization of $(A_0,\Phi_0)$.
\end{abstract}

\maketitle


%
%
\section{Introduction}

The moduli space of solutions to Hitchin's self-duality equations on a compact Riemann surface, by its   definition primarily an object of geometric analysis,   is intimately related to   a  number of diverse  fields such as    algebraic geometry, geometric topology and the emerging subject of higher Teichm\"uller theory. From an analytic point of view, it is the space of gauge equivalence classes of solutions to the system of first-order partial differential equations
\be
\label{he1}
\begin{array}{rcl}
F_A^{\perp}+[\Phi\wedge\Phi^*] & = & 0, \\[0.4ex]
\delb_A\Phi & =& 0
\end{array}
\ee
for a pair $(A,\Phi)$, where $A$ is a unitary connection on a hermitian vector
bundle $E$ over a Riemann surface $(\Sigma,J)$, and $\Phi$  is an $\End(E)$-valued $(1,0)$-form, the so-called Higgs field. Here $F_A^\perp$ is the trace-free part of the curvature of $A$,  a $2$-form with values in the skew-hermitian endomorphisms 
of $E$, and the adjoint Higgs field $\Phi^*$ is computed with respect to the hermitian metric on $E$. When restricted to  a suitable slice of the action by unitary gauge transformations, Eq.\  \eqref{he1} form a system of elliptic partial differential equations.  We always assume that   the genus $g$ of the closed surface $\Sigma$ is at least $2$. 
\medskip\\
The moduli space $\mathcal M$ of solutions to the   self-duality  equations, first introduced by Hitchin \cite{hi87} as  a two-dimensional reduction of the standard self-dual Yang-Mills equations in four dimensions, shows a  rich geometric structure in very different ways: as a quasi-projective variety \cite{hi87,si88,ni91}, as the phase space of a completely integrable system \cite{hi87b,hsw99}, and (in case where the rank and degree of $E$ are coprime) as a noncompact smooth manifold carrying a complete hyperk\"ahler metric $g_{\textrm{WP}}$ of Weil-Petersson type.  Concerning the second and the last point, we mention in particular 
the recent work of Gaiotto, Moore and Neitzke~\cite{gmn10,gmn13}   concerning hyperk\"ahler metrics on holomorphic integrable systems. They describe     a natural but incomplete hyperk\"ahler metric $g_0$ on $\mathcal M$ as a leading term (the semiflat metric in the language of~\cite{fr99}) plus an asymptotic series of non-perturbative corrections, which decay exponentially in the distance from some fixed point in moduli space. The coefficients of these correction terms are given there in terms of a priori divergent expressions coming from a wall-crossing formalism.  The thus completed hyperk\"ahler metric is conjectured to coincide with the above mentioned   metric $g_{\textrm{WP}}$ on moduli space. A further motivation to study this moduli space is Sen's conjecture about the $L^2$-cohomology of the monopole moduli spaces~\cite{se94} and the variant of it due to Hausel concerning  the $L^2$-cohomology of $\mathcal M$.
\medskip\\
The first study of the self-duality equations on    higher dimensional K\"ahler manifolds is due to Simpson \cite{si88,si90,si92}. As shown by Donaldson \cite{do87} for Riemann surfaces (and   extended to the higher dimensional case by Corlette \cite{co88}) the moduli space $\mathcal M$ corresponds  closely to the variety of representations  of (a central extension of) the fundamental group of $\Sigma$ into the Lie group $\SL(r,\C)$, $r = \mathrm{rk}(E)$ (see~\cite{go12} and references therein), and thus  permits to be studied by more algebraic and topological methods.
\medskip\\
In recent joint work,   Mazzeo, Wei{\ss}, Witt, and  the present author started an investigation of the large scale structure of this moduli space, resulting so far in a precise description of the profile of solutions in the limit of  ``large" Higgs fields, i.e.~when $\|\Phi\|_{L^2}\to\infty$, cf.~\cite{msww14,msww15}. Moreover, a geometric compactification of $\mathcal M$ was obtained, which consists in adding to $\mathcal M$  configurations $(A,\Phi)$, that are singular in a finite set of points  and admit an interpretation as so-called parabolic Higgs bundles.
\medskip\\
The present work fits into a broader line of research which aims at an understanding of the  various limits and degenerations of   Higgs bundle moduli spaces, one aspect of which has been worked out in the above mentioned articles. Complementary to these works,  we here undertake a first step in describing a rather  different degeneration phenomenon. While the complex structure $J$ of the underlying surface $\Sigma$ has been kept fixed in most of  the results  concerning the structure of the moduli space, we now view $\mathcal M=\mathcal M(\Sigma,J)$ as being parametrized by the  complex structure $J$.  Our  aim  is then to understand possible degenerations of $\mathcal M$ in the limit of a sequence   $(\Sigma,J_i)$ of complex surface converging to a ``noded" surface $(\Sigma,J_0)$, thus representing a  boundary point in the Deligne--Mumford compactification of Teichm\"uller moduli space.  Throughout, we restrict attention to the case of a {\bf{complex vector bundle $E$ of rank $r=2$}}. As a first result towards an understanding of this degeneration, we identify a limiting moduli  space $\mathcal M(\Sigma,J_0)$ consisting of those solutions $(A,\Phi)$ to the self-duality equations on the degenerate surface $(\Sigma,J_0)$ which  in the set $\mathfrak p$ of nodes show   so-called logarithmic (or first-order) singularities.  The precise definition of $\mathcal M(\Sigma,J_0)$  is given in Eq.\ \eqref{eq:defmodspace} below. 
\medskip\\
The main result presented here is   a gluing theorem which states that, subject to the assumptions (A1--A3) below, any solution $(A,\Phi)$ representing a point in $\mathcal M(\Sigma,J_0)$   arises as the uniform limit of a sequence of smooth solutions on $\mathcal M(\Sigma,J_i)$ as $J_i\to J_0$. To formulate these assumptions, we  choose near each node $p\in\mathfrak p$ a neighborhood $U_p$ complex isomorphic to  $\{zw=0\}$, cf.~\textsection \ref{subsect:degeneratingfamilies} for details. The  quadratic differential  $q:=\det\Phi$ is then  meromorphic with    poles of order at most two   in $\mathfrak p$, and we may assume that  near $z=0$   it  is in the standard form 
\begin{equation}\label{eq:holQD}
q=-C_{p,+}^2\frac{dz^2}{z^2}
\end{equation}
for some constant $C_{p,+}\in\C$, and similarly near $w=0$ for some constant $C_{p,-}$. As shown in \cite{bibo04} (cf.\ Lemma \ref{thm:asymptbehav}) the solution  $(A,\Phi)$ differs  near $z=0$  from some model solution $(A_{p,+}^{\mod}, \Phi_{p,+}^{\mod})$ (possibly after applying a unitary gauge transformation) by a term which decays polynomially as $|z|\searrow 0$, and similarly near $w=0$ for some  model solution $(A_{p,-}^{\mod}, \Phi_{p,-}^{\mod})$. These model solutions are   singular solutions to the self-duality equations of the form
\begin{equation*}
A_{p,\pm}^{\mod}=\begin{pmatrix}\alpha_{p,\pm}&0\\0&-\alpha_{p,\pm}\end{pmatrix}\left(\frac{dz}{z}-\frac{d\bar z}{\bar z}\right),\qquad \Phi_{p,\pm}^{\mod}=\begin{pmatrix}C_{p,\pm}&0\\0&-C_{p,\pm}\end{pmatrix} \frac{dz}{z},
\end{equation*}
where $\alpha_{p,\pm}\in\R $ and  the constants  $C_{p,\pm}\in\C  $ are  as above. We   impose   the following assumptions.
\begin{itemize}
\item[{\bf{(A1)}}]
The constants $C_{p,\pm}\neq0$ for every $p\in\mathfrak p$.
\item[{\bf{(A2)}}]
The constants $(\alpha_{p,+},C_{p,+})$ and $(\alpha_{p,-},C_{p,-})$ satisfy the matching conditions $\alpha_{p,+}=-\alpha_{p,-}$ and  $C_{p,+}=-C_{p,-}$ for every $p\in\mathfrak p$.   
\item[{\bf{(A3)}}]
The meromorphic quadratic differential $q$ has at least one simple zero.  \end{itemize}
The assumptions (A1--A3) will be   discussed in \textsection \ref{subsect:approxsolutions}. Let us point out here that working exclusively with    model solutions in diagonal form (rather than admitting   model Higgs fields or connections with non-semisimple endomorphism parts) matches exactly the main assumption of Biquard and Boalch in \cite[p.\ 181]{bibo04}. We keep this setup   in order to have their results available.  Assumption (A1), which is      a generic assumption on the meromorphic quadratic differential $q$, allows us to transform the solution $(A,\Phi)$   into   a standard model form near each   $p\in\mathfrak p$ via a complex gauge transformation close to the identity.  Assumption (A3), likewise satisfied by a generic meromorphic quadratic differential, is a weakening of the one used in the gluing construction of  \cite{msww14}, where it was required that \emph{all} zeroes of $q$ are simple. It allows us to show, very much in analogy to the case of closed Riemann surfaces, that a certain linear operator associated with the deformation complex of the self-duality equations at $(A,\Phi)$ is injective (cf.\ Lemma \ref{lemma:limitkernelnew}). This property if then used  in Theorem \ref{thm:L2est}  to  assure  absence of so-called small eigenvalues of  the linear operator governing the gluing construction to be described below. 
\medskip\\  
   We can now state our main theorem.

\begin{theorem}[Gluing theorem]\label{thm:mainthm}
Let $(\Sigma,J_0)$ be a  Riemann surface with nodes in a finite set of points $\mathfrak p\subset\Sigma$. Let   $(A_0,\Phi_0)$ be a solution of the self-duality equations  with   logarithmic singularities in $\mathfrak p$, thus representing a point in $\mathcal M(\Sigma,J_0)$. Suppose that  $(A_0,\Phi_0)$ satisfies the assumptions (A1--A3).   Let  $(\Sigma,J_i)$ be a sequence of smooth  Riemann surfaces converging uniformly (in a sense to be made precise in \textsection \ref{subsect:degeneratingfamilies}) to $(\Sigma,J_0)$. Then, for every sufficiently large $i\in\N$, there exists  a smooth solution $(A_i,\Phi_i)$ of Eq.\ \eqref{he1} on $(\Sigma,J_i)$ such that $(A_i,\Phi_i)\to (A_0,\Phi_0)$ as $i\to\infty$ uniformly on compact subsets of $\Sigma\setminus\mathfrak p$.
\end{theorem}

The article is organized as follows. After reviewing the necessary background on Higgs bundles, we describe the conformal plumbing construction which yields a    one-parameter family $\Sigma_{t}$, $t\in\C^{\times}$, of   Riemann  surfaces developing   a node $p\in\mathfrak p$ in the limit $t\to0$.   From a geometric point of view, endowing $\Sigma_{t}$ with the unique hyperbolic metric in its conformal class, it contains a long and thin hyperbolic cylinder $\mathcal C_t$   with central geodesic $c_t$  pinching off to  the point $p$ in the limit $t\to0$.  We describe a family of rotationally symmetric model solutions defined on finite cylinders $\mathcal C(R)$, $R=\left|t\right|$, in \textsection \ref{subsect:localmod}. The relevance of these local model solutions, in the case where the parameter $R=0$, is that by a result due to Biquard--Boalch (cf.~Lemma \ref{thm:asymptbehav}) any global solution on  the degenerate surface  $\Sigma_0$ is asymptotically close to it near the nodes. This fact allows us to construct approximate solutions on $\Sigma_{t}$ for every sufficiently small parameter $R=\left|t\right|>0$ from an exact solution on   $\Sigma_0$ by means of a gluing construction, cf.~\textsection \ref{subsect:approxsolutions}. In a final step, we employ a contraction mapping argument to  ``correct" these approximate solutions to exact ones. An analytic  difficulty arises here from the fact, typical for problems involving a ``stretching of necks", that the smallest eigenvalue $\lambda_1(R)$ of the relevant linearized operator $L_{R}$ converges to $0$ as $R\searrow0$. Hence the   limiting operator does not have a bounded inverse on $L^2$. We therefore need to control the rate of convergence to zero of   $\lambda_1(R)$, which  by an application of the Cappell--Lee--Miller gluing theorem is shown to be of order $\left|\log(R)\right|^{-2}$, cf.\ Theorem  \ref{thm:L2est}. To reach this conclusion we need to rule out the existence of eigenvalues of   $L_{R}$ of order less than $\left|\log(R)\right|^{-2}$ (the so called small eigenvalues). This requires a careful study of the deformation complex of the self-duality equations  associated with the singular solution $(A,\Phi)$ on $\Sigma_0$ and occupies a large part of \textsection \ref{subsect:cyloperators}. Subsequently, we show that the error terms coming from the approximate solutions are decaying to $0$ at    polynomial rate in $R$ as $R\searrow0$, which   allows for an application of a contraction mapping argument to prove the main theorem, cf.\ \textsection \ref {subsect:contractionmapping}.  Let us finally point out that there is one particular interesting instance of our gluing theorem in the context of  Michael Wolf's   Teichm\"uller theory of harmonic maps, which  is discussed in \textsection \ref{subsect:harmonicmaps}. 
\medskip\\
We leave it to a further publication to show a compactness result converse to the gluing theorem presented here, i.e.~to give conditions which assure that a given sequence of solutions $(A_i,\Phi_i)$ on a degenerating family of Riemann surfaces $(\Sigma,J_i)$ subconverges to a singular solution of the type discussed here. Once this analytical picture is completed, we can proceed further and study the more  geometric aspects  revolving around the family of complete hyperk\"ahler metrics on $\mathcal M(\Sigma,J_i)$ and its behavior  under   degeneration, hence paralleling the line of research initiated in \cite{msww14}.

\bigskip

\centerline{\textbf{Acknowledgements}}

\smallskip

The author wishes to thank Rafe Mazzeo, Hartmut Wei{\ss} and Frederik Witt for helpful comments and  useful discussions. He is grateful to the referees for their valuable comments, which helped to improve the exposition substantially.

%
%
\section{Moduli spaces of Higgs bundles}\label{sect:modulispaces}
In this section we review some relevant background material. A more complete introduction can be found, for example, 
in the appendix in~\cite{wgp08}.  For generalities on hermitian holomorphic vector bundles see ~\cite{ko87}. 
%

%
\subsection{Hitchin's equations}\label{sect:hitchineq}
Let $\Sigma$ be a smooth Riemann surface. We fix a hermitian vector bundle $(E,H)\to \Sigma$ of rank $2$ and degree $d(E)\in\Z$. The background hermitian metric $H$ will be used in an auxiliary manner; since any two  hermitian metrics are complex gauge equivalent the precise choice is immaterial. We furthermore fix a K\"ahler metric on $\Sigma$ such that the associated K\"ahler form $\omega$ satisfies $\int_{\Sigma}\omega=2\pi$. The main object of this article are moduli spaces of  solutions $(A,\Phi)$ of Hitchin's self-duality equations~\cite{hi87} 
\be
\label{hit.equ.uni}
\begin{array}{rcl}
F_A+[\Phi\wedge\Phi^*] & = & -i\mu(E)\operatorname{id}_E\omega, \\[0.4ex]
\delb_A\Phi & =& 0
\end{array}
\ee
for a unitary connection  $A\in\mc U(E)$ and  a {\em Higgs field} $\Phi \in \Omega^{1,0}(\End(E))$. We here denote by $\mu(E)=d(E)/2$ the slope of the rank-$2$ vector bundle $E$.
\medskip\\
The   group $\Gamma(\U(E))$ of unitary gauge transformations acts on  connections $A\in\mc U(E)$ in the usual way as $A\mapsto g^{\ast}A=g^{-1}Ag+g^{-1} dg$ and on  Higgs fields by conjugation $\Phi\mapsto g^{-1}\Phi g$.  Thus the  solution space of Eq.\ \eqref{hit.equ.uni} is preserved by $\Gamma(\U(E))$ acting diagonally on pairs $(A,\Phi)$. Moreover, the second equation  in Eq.\ \eqref{hit.equ.uni} implies that any solution $(A,\Phi)$ determines a {\em Higgs bundle} $(\delb,\Phi)$, i.e.\ a holomorphic structure $\delb=\delb_A$ on $E$ for which $\Phi$ is holomorphic: $\Phi\in H^0(\Sigma,\End(E)\otimes K_{\Sigma})$, $K_{\Sigma}$ denoting the canonical bundle of $\Sigma$.   Conversely, given a Higgs bundle $(\delb,\Phi)$, the operator $\delb$ can be augmented to a unitary connection $A$ such that the first Hitchin equation holds provided $(\delb,\Phi)$ is {\em stable}. Stability here means that $\mu(F)<\mu(E)$ for any nontrivial  $\Phi$-invariant holomorphic subbundle $F$, that is, $\Phi(F)\subset F\otimes K$.
\medskip\\
According to the Lie algebra splitting $\mf{u}(2)\cong \mf{su}(2)\oplus \mf{u}(1)$ into trace-free and pure trace summands, the bundle $\mf{u}(E)$
splits as $\mf{su}(E)\oplus i\underline{\R}$. Consequently, the curvature $F_A$ of a unitary connection  $A$ decomposes as
\begin{equation*}
	F_A = F_A^\perp+\frac{1}{2}\Tr(F_A)\otimes\operatorname{id}_E,
\end{equation*}
where $F_A^\perp\in\Omega^2(\mf{su}(E))$ is its   trace-free part  and $\frac{1}{2}\Tr(F_A) \otimes\operatorname{id}_E$ is the   pure trace  or   central part, see e.g.~\cite{po92}. Note that $\Tr(F_A) \in \Omega^2(i\underline{\R})$ equals to the curvature of the induced connection on $\det E$. From now, we fix a background connection $A_0\in\mc U(E)$   and consider only those connections $A$ which induce the same connection 
on $\det E$ as $A_0$ does. Equivalently, such a connection $A$ is of the form $A=A_0 + \alpha$ where $\alpha\in\Omega^1(\mf{su}(E))$, i.e.~$A$ is trace-free ``relative'' to $A_0$. Rather than Eq.\ \eqref{hit.equ.uni} we are from now on studying the slightly easier  system of equations 
\be\label{eq:hitequ}
\begin{array}{rcl}
F_A^\perp+[\Phi\wedge\Phi^*] & = & 0, \\[0.4ex]
\delb_A\Phi & =& 0
\end{array}
\ee
for $A$ trace-free relative to $A_0$ and a trace-free Higgs fields $\Phi\in \Omega^{1,0}(\mathfrak{sl}(E))$. There always exists a unitary connection $A_0$ on $E$ such that $\Tr F_{A_0}=-i\deg(E)\omega$. With this choice of a  background connection, any solution of Eq.\ \eqref{eq:hitequ} provides a solution to Eq.\ \eqref{hit.equ.uni}.  The relevant  groups of gauge transformations in this fixed determinant case are   $\Gamma(\SL(E))$ and $\Gamma(\SU(E))$, the former being the complexification of the latter, which we denote by $\mc G^c$ and $\mc G$ respectively.

\subsection{The degenerating family of Riemann surfaces}\label{subsect:degeneratingfamilies}
%

We recall the well-known plumbing construction for Riemann surfaces, our exposition here following largely \cite{ww92}. A {\bf{Riemann surface with nodes}} is  a  one-dimensional complex analytic space $\Sigma_0$ where each point has a neighborhood complex isomorphic to a disk $\{|z|<\epsilon\}$ or to $U=\{zw=0\mid \left|z\right|,\left|w\right|<\epsilon\}$, in which case it is called a node. A Riemann surface with  nodes arises from an unnoded surface by pinching of one or more  simply closed curves. Conversely, the effect of  the so-called conformal plumbing construction  is that it opens up a node by replacing the neighborhood $U$ by $\{zw=t\mid t\in\C,\left|z\right|,\left|w\right|<\epsilon\}$. To describe this construction in more detail, let $(\Sigma_0,z,p)$ be a Riemann surface of genus $g\geq2$ with conformal coordinate $z$ and a single node at $p$. Let $t\in\C\setminus\{0\}$ be fixed with $\left|t\right|$ sufficiently small. We then  define a smooth Riemann surface $\Sigma_t$ by removing the disjoint disks $D_t=\{\left|z\right|<\left|t\right|,\left|w\right|<\left|t\right|\}\subseteq U$ from $\Sigma_0$ and passing to the quotient space $\Sigma_t=(\Sigma_0\setminus D_t)/_{zw=t}$, which  is a Riemann surface of the same genus as $\Sigma_0$.   We allow for Riemann surfaces with a finite number of nodes, the set of which we denote $\mathfrak p=\{p_1,\ldots, p_k\}\subset \Sigma_0$, and impose the assumption that $ \Sigma_0\setminus  \mathfrak p$ consists of $k+1$ connected components
.  To deal with the case of multiple nodes in an efficient way we make the {\bf{convention}} that in the notation $\Sigma_t$ the dependence of  $t\in\C$ on the point  $p\in\mathfrak p$ is suppressed. The value of $t$ may be different at different nodes. Also, $\rho=\left|t\right|$ refers to the maximum of these   absolute values.   
\medskip\\
We endow each Riemann surface $\Sigma_t$ with a Riemannian metric compatible with its complex structure in the following way. Let $\rho=\left|t\right|<1$ and consider the annuli
\begin{equation}\label{eq:annuli}
R_{\rho}^+=\{z\in\C\mid \rho\leq\left|z\right|\leq 1\}\qquad\textrm{and}\qquad  R_{\rho}^-=\{w\in\C\mid \rho\leq\left|w\right|\leq 1\}.
\end{equation}
The above identification of $R_{\rho}^+$ and $R_{\rho}^-$ along their inner boundary circles $\{\left|z\right|=\rho\}$ and $\{\left|w\right|=\rho\}$ yields a smooth cylinder $C_{t}$. The Riemannian metrics 
\begin{equation*}
g^+=\frac{\left|dz\right|^2}{\left|z\right|^2},\qquad\textrm{respectively}\qquad g^-=\frac{\left|dw\right|^2}{\left|w\right|^2}
\end{equation*}
on $R_{\rho}^{\pm}$ induce a smooth metric on $C_{t}$, which we extend smoothly over $\Sigma_t$ to a metric compatible with the complex structure.  Note that the cylinder $C_{t}$ endowed with this metric   is flat. Indeed, the map $(r,\theta)\mapsto(\tau,\vartheta):= (-\log r,-\theta)$ provides an isometry between  $(R_{\rho}^+,g^+)$ and the standard flat cylinder $[0,-\log\rho]\times S^1$ with metric $d\tau^2+ d\vartheta^2$, and similarly for $(R_{\rho}^-,g^-)$.
\medskip\\
For  a closed Riemann surface $\Sigma$ recall    the space $\operatorname{QD}(\Sigma)$ of holomorphic quadratic differentials, which by definition is the $\C$-vector space of   holomorphic sections $q$ of the bundle $K_{\Sigma}^2$ of complex-valued symmetric bilinear forms. Its elements    are locally of the form $q=u\,dz^2$ for some holomorphic function $u$.  By the Riemann-Roch theorem, its complex dimension is $\dim \operatorname{QD}(\Sigma)=3(g-1)$, cf.~\cite[Appendix F]{tr92}. On a noded Riemann surface we will allow for quadratic differentials meromorphic  with poles of order at most $2$ at points in the subset $\mathfrak p\subset \Sigma$ of nodes.  In this case,  the corresponding $\C$-vector space of meromorphic quadratic differentials is denoted by $\operatorname{QD}_{-2}(\Sigma)$.
\medskip\\
Setting $\rho=0$ in Eq.\ \eqref{eq:annuli} yields for each $p\in\mathfrak p$ a punctured neighborhood $\mathcal C_0\subset\Sigma_0$ of $p$ consisting of two connected components $\mathcal C_0^{\pm}$. We endow these with   cylindrical coordinates $(\tau^{\pm},\vartheta^{\pm})$ as before.  Together with the above chosen Riemannian metric this turns $\Sigma_0$ into a manifold with cylindrical ends. We in addition fix a  hermitian vector bundle $(E,H)$  of rank $2$ over $\Sigma_0$, which we suppose is cylindrical in the sense to be described in \textsection \ref{subsect:cyloperators}. Briefly, this means that the restrictions of $(E,H)$ to  $\mathcal C_0^{\pm}$   are invariant under pullback by translations in the $\tau^{\pm}$-directions. We furthermore require that the restriction of  $(E,H)$ to $\mathcal C_0$ is invariant under pullback via the (orientation reversing) isometric involution $(\tau^{\pm},\vartheta^{\pm})\mapsto(\tau^{\mp},\operatorname{arg}t-\vartheta^{\mp})$ interchanging the two half-infinite cylinders $\mathcal C^+$ and $\mathcal C^-$. The pair $(E,H)$ induces a hermitian vector bundle on the surface $\Sigma_t$ by restriction, which extends smoothly over the cut-locus $\left|z\right|= \left|w\right|=\rho$.

\begin{figure}   
\label{bild1}                                                    
\includegraphics[width=1.1\textwidth]{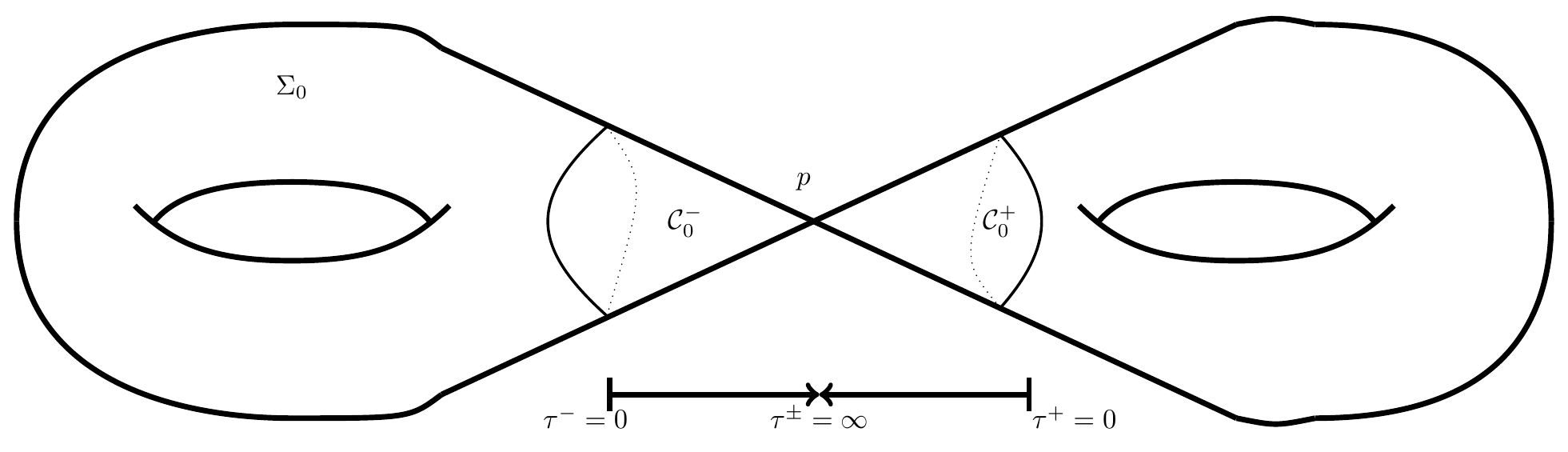}                         
\caption{Degenerate Riemann surface $\Sigma_0$ with one node $p\in\mathfrak p$, which separates the two half-infinite cylinders $\mathcal C_0^{\pm}$.}
\end{figure}

%
%
\subsection{The local model}\label{subsect:localmod}
 
Fix constants $\alpha\in\R$ and $C\in\C$. Then the pair
\begin{equation}\label{eq:modsol}
A^{\mod}=\begin{pmatrix}\alpha&0\\0&-\alpha\end{pmatrix}\left(\frac{dz}{z}-\frac{d\bar z}{\bar z}\right),\quad \Phi^{\mod}=\begin{pmatrix}C&0\\0&-C\end{pmatrix}\frac{dz}{z}
\end{equation}
provides a  solution on $\C^{\ast}$ to  Eq.\ \eqref{eq:hitequ}, which we call {\bf{model solution to parameters}} $(\alpha,C)$. It is smooth outside the origin and has a logarithmic (first-order) singularity in $z=0$, provided that $\alpha$ and $C$ do not both vanish. It furthermore restricts to a smooth solution on each of the annuli $R_{\rho}^{\pm}$ defined in Eq.\ \eqref{eq:annuli}. Note that the parameters encountered in the example of \textsection \ref{subsect:harmonicmaps} below are $(\alpha,C)=(0,\frac{\ell}{2i})$. We also remark that we do not consider the more general solution with connection
\begin{equation*}
A=\begin{pmatrix}\alpha&0\\0&-\alpha\end{pmatrix}\frac{dz}{z}-\begin{pmatrix}\bar\alpha&0\\0&-\bar\alpha\end{pmatrix}\frac{d \bar z}{\bar z}
\end{equation*}
for   parameter $\alpha\in\C$, because it is unitarily  gauge equivalent to the above model solution. Since
\begin{equation*}
\frac{dz}{z}-\frac{d\bar z}{\bar z}=2i\,d\theta,
\end{equation*}
the connections $A^{\mod}$ appearing in Eq.\ \eqref{eq:modsol} are in radial gauge, i.e.\ their $dr$-components vanish identically. For constants $t\in\C$ and $\rho=\left| t\right|$ such that $0<\rho<1$ let $\mathcal C_t$ denote the complex cylinder obtained from gluing the two annuli $R_{\rho}^-$ and $R_{\rho}^+$. Since
\begin{equation*}
\frac{dz}{z}=-\frac{dw}{w}
\end{equation*} 
the two model solutions $(A_+^{\mod},\Phi_+^{\mod})$  to parameters $(\alpha,C)$ over $R_{\rho}^+$ and  $(A_-^{\mod},\Phi_-^{\mod})$ to parameters $(-\alpha,-C)$ over $R_{\rho}^-$ glue to a smooth solution $(A^{\mod},\Phi^{\mod})$ on $\mathcal C_t$, again called model solution to parameters $(\alpha,C)$. In the following it is always assumed that $\alpha>0$.

\subsection{Analytical setup}\label{subsect:parabHiggs}
%

We  introduce weighted Sobolev spaces of connections and gauge transformations over the punctured Riemann surface $\Sigma_0$, following the analytical setup of \cite{bibo04}.
\medskip\\
Let $r$ be a strictly positive function on $\Sigma_0$ such that  $r=\left|z\right|$ near the subset  $\mathfrak p\subset\Sigma_0$ of nodes. The weighted Sobolev spaces to be introduced next are all defined with respect to the measure $r\,dr\,d\theta$ on $\C$ which we abbreviate as $r\,dr$. Note that later on we will have cause to  use also the more singular measure $r^{-1}\,dr\,d\theta$.  For a weight $\delta\in\R$ we define the $L^2$ based weighted Sobolev spaces 
\begin{equation*}
L_{\delta}^2=\left\{u\in L^2(r\,dr)\mid r^{-\delta-1}u\in L^2(r\,dr)\right\}
\end{equation*}
and
\begin{equation*}
H_{\delta}^k =\left\{ u, \nabla^ju\in L_{\delta}^2(r\,dr), 0\leq j\leq k\right\},
\end{equation*}
with the unitary background connection $\nabla$ on $E$ being fixed throughout. Note that the function $r^{\beta}\in L_{\delta}^2$ if and only if $\beta>\delta$. The spaces $H_{\delta}^k$ satisfy a number of standard Sobolev embedding and multiplication theorems such as 
\begin{equation*}
H_{-2+\delta}^k\cdot H_{-2+\delta}^k\subseteq H_{-2+\delta}^k\quad (k\geq2),\qquad H_{-2+\delta}^2\cdot H_{-2+\delta}^1\subseteq H_{-2+\delta}^1, 
\end{equation*}
cf.~\cite[\textsection 3]{bibo04} for details. 
\medskip\\
From now on let a constant $\delta>0$ be fixed. We shall work with the  spaces of unitary connections and Higgs fields
\begin{equation*}
\mathcal A_{-2+\delta}^1=\left\{A^{\mod}+\alpha\mid \alpha\in H_{-2+\delta}^1(\Omega^1\otimes \mathfrak{su}(E))\right\}
\end{equation*}
and
\begin{equation*}
\mathcal B_{-2+\delta}^1=\left\{\Phi^{\mod}+\varphi\mid \varphi\in H_{-2+\delta}^1(\Omega^{1,0}\otimes \mathfrak{sl}(E))\right\},
\end{equation*}
where $(A^{\mod},\Phi^{\mod})$ is some singular model solution on $\Sigma_0$ as introduced in Eq.\ \eqref{eq:modsol}. It is important to note that we have to allow for  arbitrary small weights $\delta>0$  as  the example   discussed in \textsection\ref{subsect:harmonicmaps} shows. These spaces are acted on smoothly by the Banach Lie group
\begin{equation*}
\mathcal G=\mathcal G_{-2+\delta}^2=\left\{ g\in\SU(E)\mid g^{-1}dg\in H_{-2+\delta}^1(\Omega^1\otimes \mathfrak{su}(E))\right\}
\end{equation*}
of special unitary gauge transformations, where
\begin{equation*}
g^{\ast}(A,\Phi)=(g^{-1}Ag+g^{-1}dg,g^{-1}\Phi g)
\end{equation*}
for $(A,\Phi)\in\mathcal A_{-2+\delta}^1\times\mathcal B_{-2+\delta}^1$, cf.\ \cite[Lemma 2.1]{bibo04}. The complexification of $\mathcal G_{-2+\delta}^2$ is the Banach Lie group
\begin{equation*}
\mathcal G^c=\mathcal G_{-2+\delta}^{2,c}=\left\{ g\in\SL(E)\mid g^{-1}dg\in H_{-2+\delta}^1(\Omega^1\otimes \mathfrak{sl}(E))\right\}.
\end{equation*}
The moduli space   of key interest in this article  is then defined to be  the quotient
\begin{equation}\label{eq:defmodspace}
\mathcal M(\Sigma_0)=\frac{\left\{(A,\Phi)\in\mathcal A_{-2+\delta}^1\times\mathcal B_{-2+\delta}^1\mid (A,\Phi)\;\textrm{satisfies Eq.}\;\eqref{hit.equ.uni}\right\}}{\mathcal G_{-2+\delta}^2},
\end{equation}
along with  the analogously defined spaces $\mathcal M(\Sigma_{t})$ for parameter $t\neq0$.

\subsection{Guiding example:  Teichm\"uller theory of harmonic maps}\label{subsect:harmonicmaps}
%

Due to Hitchin, there is a close relation between Teich\-m\"ul\-ler theory and the moduli space of $\SL(2,\R)$-Higgs bundles which we recall next. Let $\Sigma$ be a Riemann  surface of genus $g$ with associated canonical bundle $K_{\Sigma}\cong \Omega^{1,0}(\Sigma)$. Let $L$ be a holomorphic line bundle of  degree $g-1$ such that $L^2\cong K_{\Sigma}$ and set $E=L\oplus L^{-1}$.   With respect to this splitting we define the Higgs field
\begin{equation*}
\Phi=\begin{pmatrix}0&q\\1&0\end{pmatrix}\in\Omega^{1,0}(\Sigma,\mathfrak{sl}(E))
\end{equation*}
for some  fixed   holomorphic quadratic differential
\begin{equation*}
q\in\Omega^{1,0}(\Sigma,\Hom(L^{-1},L))\cong\operatorname{QD}(\Sigma)
\end{equation*}
and the constant function $1\in\Omega^{1,0}(\Sigma,\Hom(L,L^{-1}))\cong \Omega^0(\Sigma,\C)$. Clearly, $\bar\partial_E\Phi=0$. We also  endow the holomorphic line bundle $L$ with an auxiliary hermitian metric  $h_0$. It induces on $E$ the hermitian metric $H_0=h_0\oplus h_0^{-1}$; let  $A=A_L\oplus A_{L^{-1}}$ denote the associated Chern connection. As one can show,    the Higgs bundle $(E,\Phi)$ is stable in the terminology of \cite{hi87}. Therefore, there exists a complex gauge transformation  $g\in\mathcal G^c$, unique up to modification by a unitary gauge transformation, such that $(A_1,\Phi_1):=g^{\ast}(A,\Phi)$ is a solution of   Eq.\ \eqref{eq:hitequ}. We argue that in this particular case the self-duality equations reduce to a scalar PDE on $\Sigma$. 
\medskip\\
To this aim, we first  modify $\Phi_1$ by a locally defined unitary gauge transformation $k$ such that the Higgs field  $\Phi_2:=(gk)^{-1}\Phi gk$ has again off-diagonal form. Set $g_1:=gk$ and $A_2:=g_1^{\ast}A$. By invariance of the second of the self-duality equations under complex gauge transformations,  $\bar\partial_{A_2} \Phi_2=0$. This equation can only hold true if the connection $A_2$ is diagonal  for otherwise  a nonzero diagonal term  would arise. 
One can furthermore check that the local complex gauge transformation $g_1$ mapping $(A,\Phi)$ to $ (A_2,\Phi_2)$ is diagonal, and hence so  is the hermitian metric $H_1:=H_0(g_1g_1^{\ast})^{-1}= H_0(g^{\ast})^{-1} (k^{\ast})^{-1}k^{-1}g^{-1}=H_0(g^{\ast})^{-1}g^{-1}$. Since this expression is independent of the above chosen local gauge transformation $k$ we conclude that $H_1$ is globally well-defined and diagonal with respect the holomorphic splitting of $E$ into line bundles. We set $H_1=h_1\oplus h_1^{-1}$, where  $h_1=e^{2u}h_0$ for some smooth function $u\colon\Sigma\to\R$.  As a standard fact, the connection $A_1$ is the Chern connection with respect to the hermitian metric $H_1$. It  thus  splits into   $A_1=A_{1,L}\oplus A_{1,L^{-1}}$. The component   $A_{1,L}$  has curvature $F_{A_{1,L}}=F_{A_L}-2\bar\partial\partial u$. Furthermore, the first diagonal entry of the  commutator term $[g^{-1}\Phi g \wedge(g^{-1}\Phi g)^{\ast}]$ equals to $e^{4u}h_0^2q\bar q-e^{-4u}h_0^{-2}$. Hence the first of the self-duality equations   reduces to the scalar PDE
\begin{equation}\label{eq:exselfdual}
 2\bar\partial\partial u-e^{4u}h_0^2q\bar q+e^{-4u}h_0^{-2}- F_{A_L}=0
\end{equation}
for the function  $u$.
\medskip\\
A   calculation furthermore shows that Eq.\ \eqref{eq:exselfdual}  is satisfied if and only if the Riemannian metric
\begin{equation}\label{eq:connstcurvmetr}
G=q+(h_1^{-2}+h_1^2q\bar q)+\bar q\in \operatorname{Sym}^2(\Sigma)
\end{equation}
has constant Gauss curvature equal to $-4$, cf.\ \cite{hi87,dw07} for further details. From a different perspective, one can interpret Eq.\ \eqref{eq:exselfdual} as the condition that the identity map   between the surface $\Sigma$ with its given conformal structure and the   surface $(\Sigma,G)$ is harmonic. This point of view has been pursued by Wolf in \cite{wo89},  where it was shown that  for every $q\in\operatorname{QD}(\Sigma)$ there is a unique solution of Eq.\ \eqref{eq:exselfdual}, therefore leading to a different proof of Teich\-m\"ul\-ler's well-known theorem stating that the Teich\-m\"ul\-ler moduli space $\mathcal T_g$ is diffeomorphic to a cell.   Wolf in \cite{wo91} then studied the behavior of this harmonic map under degeneration of the domain Riemann surface $\Sigma$ to a noded surface $\Sigma_0$. Let us now describe those aspects of his theory which are the most relevant to us.
\medskip\\
Let $\mathcal C=(S^1)_x\times[1,\infty)_y$ denote the half-infinite cylinder, endowed with the complex coordinate $z=x+iy$ and flat Riemannian metric $g_{\mathcal C}=\left|dz\right|^2=dx^2+dy^2$. Furthermore, for parameter $\ell>0$ let 
\begin{equation*}
N_{\ell}=[\ell^{-1}\csc^{-1}(\ell^{-1}),\pi/\ell-\ell^{-1}\csc^{-1}(\ell^{-1})]_u\times (S^1)_v
\end{equation*}
be the finite cylinder with complex coordinate $w=u+iv$. It carries the hyperbolic metric $g_{\ell}=\ell^2\csc^2(\ell u)\left|dw\right|^2$. In \cite{wo91}, Wolf discusses the one-parameter family of infinite-energy harmonic maps
\begin{equation*}
w_{\ell}\colon (\mathcal C,\left|dz\right|^2)\to (N_{\ell},g_{\ell}),\qquad w_{\ell}=u_{\ell}+iv_{\ell},
\end{equation*}
where
\begin{equation*}
v_{\ell}(x,y)=x,\qquad u_{\ell}(x,y)=\frac{1}{\ell}\sin^{-1}\left(\frac{1-B_{\ell}(y)}{1+B_{\ell}(y)}\right),
\end{equation*}
and
\begin{equation*}
B_{\ell}(y)=\frac{1-\ell}{1+\ell}e^{2\ell(1-y)}.
\end{equation*}
It serves as a model for harmonic maps with domain a noded Riemann surface and target  a smooth Riemann surface containing a long hyperbolic  ``neck" with central geodesic of length $2\pi\ell$. Indeed, Wolf (cf.~\cite[Proposition 3.8]{wo91}) shows that the unique harmonic such map $w_{\ell}$ homotopic to the identity  is exponentially close to the above model harmonic map. The pullback to $\mathcal C$ of the metrics $g_{\ell}$ yields the family of hyperbolic metrics
\begin{eqnarray*}
G_{\ell}&=&\ell^2\left(\frac{1+B_{\ell}}{1-B_{\ell}}\right)^2\left(dx^2+\frac{4B_{\ell}}{(1+B_{\ell})^2}dy^2\right)\\
&=&\ell^2\left(\frac{1+B_{\ell}}{1-B_{\ell}}\right)^2\left(\left(\frac{1}{4}-\frac{B_{\ell}}{(1+B_{\ell})^2}\right)\,dz^2+\left(\frac{1}{2}+\frac{2B_{\ell}}{(1+B_{\ell})^2}\right)\,dz\,d\bar z\right.\\
&&\left.+\left(\frac{1}{4}-\frac{B_{\ell}}{(1+B_{\ell})^2}\right)\,d\bar z^2\right).
\end{eqnarray*}
Since $w_{\ell}$ is  harmonic, the $dz^2$ component $q_{\ell}$ of $G_{\ell}$ is a holomorphic quadratic differential on $\mathcal C$. In the framework described above, it is induced by the Higgs field
\begin{equation*}
\Phi_{\ell} =\begin{pmatrix}0&q_{\ell}\\
1&0\end{pmatrix}. 
\end{equation*}
We choose a local holomorphic trivialization of $E$ and suppose that with respect to it the auxiliary hermitian metric $h_0$ is the standard hermitian metric on $\C^2$. Comparing the metric $G_{\ell}$ with the one in Eq.\ \eqref{eq:connstcurvmetr} it follows that the hermitian metric $h_{1,\ell}$ on $L$ in this particular example is
\begin{equation*}
h_{1,\ell}=\frac{2}{\ell} \frac{1-B_{\ell}^{\frac{1}{2}}}{1+B_{\ell}^{\frac{1}{2}}}.
\end{equation*}
The corresponding hermitian metric on $E=L\oplus L^{-1}$ is 
\begin{equation*}
H_{1,\ell}=\begin{pmatrix}h_{1,\ell} &0\\0&h_{1,\ell}^{-1}\end{pmatrix}\end{equation*}
and thus any complex gauge transformation
\begin{equation*}
g_{\ell}=\begin{pmatrix}e^{-u_{\ell}}&0\\0&e^{u_{\ell}}\end{pmatrix}
\end{equation*}
satisfying $g_{\ell}^2=H_{1,\ell}^{-1}$ gives rise to a solution of Eq.\ \eqref{eq:exselfdual}, as one may check by direct calculation. With respect to the above chosen  local holomorphic trivialization of $E$, the Chern connection $A$   becomes the trivial connection, and thus the solution of the self-duality equations corresponding to the Higgs pair $(A,\Phi_{\ell})$ equals $(A_{1,\ell},\Phi_{1,\ell})=g_{\ell}^{\ast}(0,\Phi_{\ell})$.
We change complex coordinates to
\begin{equation*}
\zeta=e^{iz},\qquad i\,dz=\frac{d\zeta}{\zeta},
\end{equation*}
so to make the resulting expressions easier to compare with the setup in \textsection \ref{sect:localmod}. Under this coordinate transformation the cylinder $\mathcal C$ is mapped conformally to the punctured unit disk. We then obtain that
\begin{equation*}
A_{1,\ell} =\frac{\ell}{2}\frac{B_{\ell}^{\frac{1}{2}}}{1-B_{\ell}}  \begin{pmatrix}1&0\\0&-1\end{pmatrix}\left(\frac{d\zeta}{\zeta}-\frac{d\bar \zeta}{\bar \zeta}\right)
\end{equation*}
and
\begin{equation*}
\Phi_{1,\ell} =\begin{pmatrix}0&\frac{\ell^2}{4}h_{1,\ell}\\
h_{1,\ell}^{-1}&0\end{pmatrix} \frac{d\zeta}{i\zeta}.
\end{equation*}
Since $B_{\ell}(\zeta)= \frac{1-\ell}{1+\ell}e^{2\ell}|\zeta|^{2\ell}$ and $h_{\ell}=\frac{2}{\ell}(1+\mathcal O(|\zeta|^{\ell})$,  it follows that
\begin{equation*}
A_{1,\ell}=\mathcal O(|\zeta|^{\ell})\begin{pmatrix}1&0\\0&-1\end{pmatrix}\left( \frac{d\zeta}{\zeta}-\frac{d\bar \zeta}{\bar \zeta}\right),\qquad  \Phi_{1,\ell}=(1+\mathcal O(|\zeta|^{\ell})\begin{pmatrix}0&\frac{\ell}{2} \\
\frac{\ell}{2}&0\end{pmatrix} \frac{d\zeta}{i\zeta}.
\end{equation*}
Therefore, after a unitary change of frame, the Higgs field $\Phi_{1,\ell}$ is asymptotic to the model Higgs field  
\begin{equation*}
\Phi_{\ell}^{\mod}=\begin{pmatrix}\frac{\ell}{2}&0\\0&-\frac{\ell}{2}\end{pmatrix}\frac{d\zeta}{i\zeta},
\end{equation*}
while the connection $A_{1,\ell}$ is asymptotic to the trivial flat connection. This ends the discussion of the example.

%
%
\section{Approximate solutions}\label{sect:localmod}

\subsection{Approximate solutions}\label{subsect:approxsolutions}
%

Throughout this section, we  fix a constant $\delta>0$ and a solution $(A,\Phi)\in \mathcal A_{-2+\delta}^1\times \mathcal B_{-2+\delta}^1$ of    Eq.\ \eqref{eq:hitequ} on the noded Riemann surface $\Sigma_0$. It represents a point in the moduli space $\mathcal M(\Sigma_0)$ as defined through Eq.\ \eqref{eq:defmodspace}. Our immediate next goal is to construct from $(A,\Phi)$  a good approximate solution on each ``nearby'' surface $\Sigma_t$ (where $\rho=\left|t\right|$ is small). We do so by  interpolating between the model solution $(A^{\mod},\Phi^{\mod})$ on each punctured cylinder $\mathcal C_p(0)$, $p\in\mathfrak p$,   and   the exact solution $(A,\Phi)$ on the ``thick" region of $\Sigma_0$, i.e.~the complement of the union  $\bigcup_{p\in\mathfrak p}\mathcal C_p(0)$. We therefore obtain an approximate solution on $\Sigma_0$, which induces one on each of the surfaces  $\Sigma_t$ obtained from $\Sigma_0$ by performing the conformal plumbing construction of \textsection \ref{subsect:degeneratingfamilies}. We also rename the  cylinder $\mathcal C_t$ obtained in the plumbing step at $p\in\mathfrak p$ (cf.\ Eq.\ \eqref{eq:annuli}) to $\mathcal C_p(\rho)$ to indicate that its length is asymptotically equal to $\left|\log\rho\right|$.
\medskip\\
Near each node $p\in\mathfrak p$ we  choose a neighborhood $U_p$ complex   isomorphic to  $\{zw=0\}$, cf.~\textsection \ref{subsect:degeneratingfamilies}. Since $(A,\Phi)$ satisfies the second equation in Eq.\ \eqref{eq:hitequ} it follows that $q:=\det\Phi$ is a meromorphic quadratic differential, its set of  poles being contained in $\mathfrak p$. As we are here only considering model solutions with poles of order one,   the poles of $q$ are  at most of order two, hence $q$ is an element of $\operatorname{QD}_{-2}(\Sigma_0)$.  After a holomorphic change of coordinates, we may assume that  near $z=0$ (and similarly near $w=0$) the meromorphic quadratic differential $q=\det \Phi$  is in the standard form 
\begin{equation}\label{eq:holQD}
q=-C_{p,+}^2\frac{dz^2}{z^2}
\end{equation}
for some constant $C_{p,+}\in\C$, and similarly near $w=0$ for some constant $C_{p,-}$. By definition of the space $\mathcal A_{-2+\delta}^1\times \mathcal B_{-2+\delta}^1$ there exists for each $p\in\mathfrak p$ a model solution  $(A_{p,\pm}^{\mod},\Phi_{p,\pm}^{\mod})$ such that (after applying a suitable unitary gauge transformation in $ \mathcal G$ if necessary) the solution $(A,\Phi) $ is asymptotically close to it near $p$. This model solution takes the form  
\begin{equation*}
A_{p,+}^{\mod}=\begin{pmatrix}\alpha_{p,+}&0\\0&-\alpha_{p,+}\end{pmatrix}\left(\frac{dz}{z}-\frac{d\bar z}{\bar z}\right)
\end{equation*}
for   some constant $\alpha_{p,+}\in\R $ and 
\begin{equation*}
\Phi_{p,+}^{\mod}=\begin{pmatrix}C_{p,+}&0\\0&-C_{p,+}\end{pmatrix} \frac{dz}{z}, 
\end{equation*}
where  the constant $C_{p,+}\in\C  $ is  as above (and similarly for $(A_{p,-}^{\mod},\Phi_{p,-}^{\mod})$). To be able to carry out the construction of approximate solutions, we need to  impose  on  $(A_{p,\pm}^{\mod},\Phi_{p,\pm}^{\mod})$  the     {\bf{assumptions (A1--A3)}} as stated in the introduction. Let us recall them here and   discuss their significance.
\begin{itemize}
\item[{\bf{(A1)}}]
The constants $C_{p,\pm}\neq0$ for every $p\in\mathfrak p$.  This is a generic condition on the   meromorphic quadratic differential $q$ in Eq.\ \eqref{eq:holQD}. It is imposed here to assure that near $z=0$ (respectively near $w=0$) the Higgs field $\Phi$ is of the form $\Phi(z)=\varphi(z)\frac{dz}{z}$ for some \emph{diagonalizable} endomorphism $\varphi(z)$. We use this assumption in Proposition \ref{prop:diaggauge} to transform $(A,\Phi)$ via a complex gauge transformation $g$ (which we   show can be chosen to be close to the identity) into $g^{\ast}(A,\Phi)$, where $g^{\ast}A$ is diagonal and $g^{-1}\Phi g$ coincides with $\Phi_{p,\pm}^{\mod}$.  For a second time,  the assumption (A1) is used in the proof of Proposition \ref{prop:asympttrace}. Together with assumption (A3) it implies injectivity  of the operator $L_1+L_2^{\ast}$ associated with the deformation complex of the self-duality equations at the solution $(A,\Phi)$.
\item[{\bf{(A2)}}]
The constants $(\alpha_{p,+},C_{p,+})$ and $(\alpha_{p,-},C_{p,-})$ satisfy the matching conditions $\alpha_{p,+}=-\alpha_{p,-}$ and  $C_{p,+}=-C_{p,-}$ for every $p\in\mathfrak p$.   
 Thus in particular, the coefficients of $z^{-2}$ and $w^{-2}$ in Eq.\ \eqref{eq:holQD} coincide. As discussed in \textsection \ref{subsect:localmod},  this matching condition permits us to construct from any pair of  singular model solutions $(A_{p,\pm}^{\mod},\Phi_{p,\pm}^{\mod})$ on the cylinder $\mathcal C_p(\rho)$  a   smooth model solution  using the conformal plumbing construction. 
\item[{\bf{(A3)}}]
The meromorphic quadratic differential $q=\det\Phi$ has at least one simple zero. This is a generic assumption on a  meromorphic quadratic    differential. It weakens the one imposed in  \cite{msww14}, where it was required that {\emph{all}} zeroes of $q$ are simple. As a consequence, it was shown there that     solutions $(A,\Phi)$ to the self-duality equations are necessarily irreducible (cf.\ \cite[p.\ 7]{msww14}).  The assumption (A3) is used here in a very similar way, namely to prove that the operator $L_1+L_2^{\ast}$  arising in the deformation complex of the self-duality equations is injective, cf.\ Lemma \ref{lemma:limitkernelnew}. This in particular implies injectivity of the operator $L_1$ and therefore irreducibility of the solution $(A,\Phi)$.
\end{itemize}
From now on, we rename the model solution $(A_{p,\pm}^{\mod},\Phi_{p,\pm}^{\mod})$ to $(A_{p}^{\mod},\Phi_{p}^{\mod})$ and the constants $C_{p,+}$ and $\alpha_{p,+}$ to $C_{p}$ and $\alpha_{p}$, respectively. Recall that  the solution $(A,\Phi)$ (after modifying it by a suitable unitary gauge transformation if necessary) differs from $(A_{p}^{\mod},\Phi_{p}^{\mod})$ by some element in $H_{-2+\delta}^1$. However,  according to Biquard-Boalch \cite{bibo04} it is asymptotically close to it in a much stronger sense.

\begin{lemma}\label{thm:asymptbehav}
Let $(A,\Phi)$ be a  solution of Eq.\ \eqref{eq:hitequ} and for each $p\in\mathfrak p$ let    $(A_{p}^{\mod},\Phi_{p}^{\mod})$ be the  pair of model solutions as before. Then there exists a unitary gauge transformation $g\in \mathcal G$      such that in  some  neighborhood of each point $p\in\mathfrak p$   it holds that 
\begin{equation*}
g^{\ast}(A,\Phi)=(A_{p}^{\mod}+\alpha,\Phi_{p}^{\mod}+\varphi), 
\end{equation*}
where $(\alpha,\varphi),(r\nabla\alpha,r\nabla\varphi),(r^2\nabla^2\alpha,r^2\nabla^2\varphi)\in C_{-1+\delta}^0$.    
\end{lemma}

\begin{proof}
For a proof of the first statement we refer to \cite[Lemma 5.3]{bibo04}. The statement that $(r\nabla\alpha,r\nabla\varphi),(r^2\nabla^2\alpha,r^2\nabla^2\varphi)\in C_{-1+\delta}^0$ follows similarly. However, since it is not carried out there, we add its proof for completeness. From \cite[Lemma 5.3]{bibo04} we get the existence of a unitary gauge transformation $g\in \mathcal G$ such that $g^{\ast}(A,\Phi)=(A_{p}^{\mod}+\alpha,\Phi_{p}^{\mod}+\varphi)$, where $b:=(\alpha,\varphi)$ satisfies the equation
\begin{equation*}
Lb=b \odot	b.
\end{equation*}
 Here $L$ is a first-order elliptic differential operator with constant coefficients with respect to cylindrical coordinates $(t,\vartheta)=(-\log r,-\theta)$, while   $b \odot	b$ denotes some bilinear combination of the function $b$ with constant coefficients. Differentiating this equation we obtain that
\begin{equation}\label{eq:nablab}
L(\nabla b)=b \odot	\nabla b+ [L,\nabla]b.
\end{equation}
We now follow the line of argument in  \cite[Lemma 4.6]{bibo04} to prove the claim. There it is shown that   $b\in L_{-2+\delta}^{1,p}$ for any $p>2$, with the weighted   spaces $L_{\delta}^{k,p}$ being defined in analogy to $H_{\delta}^k$. Hence $\nabla b\in L_{-2+\delta}^p$ which together with $b\in C_{-1+\delta}^0$ yields the inclusion $b \odot	\nabla b \in L_{-3+2\delta}^p$. Similarly, $[L,\nabla]b \in C_{-1+\delta}^0\subseteq L_{-3+2\delta}^p$ since   $[L,\nabla]$ is an operator of order zero with constant coefficients. It thus follows from Eq.\ \eqref{eq:nablab} and elliptic regularity   that $\nabla b \in L_{-3+2\delta}^{1,p}$. For any $p>2$ we have the  continuous embedding $L_{-3+2\delta}^{1,p}\hookrightarrow C_{-3+2\delta+2/p}^0$ (cf.\ \cite[Lemma 3.1]{bibo04}) which we apply with $p=\frac{2}{1-\delta}$ to conclude that $\nabla b \in C_{-2+\delta}^0$. Hence $r\nabla b \in C_{-1+\delta}^0$. The remaining claim   that $r^2\nabla^2 b \in C_{-1+\delta}^0$ is finally shown along the same lines. Differentiating Eq.\ \eqref{eq:nablab} leads to the equation
\begin{equation*}\label{eq:nablab1}
L(\nabla^2 b)=\nabla b \odot	\nabla b+b\odot\nabla^2b+ [L,\nabla]	\nabla b+ [\nabla L,\nabla]	 b.
\end{equation*}
By inspection, using the regularity results already proven, we have that the right-hand side is contained in $L_{-4+2\delta}^p$ for any $p>2$. Elliptic regularity yields  $\nabla^2 b\in  L_{-4+2\delta}^{1,p}$. Since for $p>2$ the latter space embeds continuously into $C_{-4+2\delta+2/p}^0$ it follows as before that  $r^2\nabla^2 b\in C_{-1+\delta}^0$, as claimed.
\end{proof}

From now on, we rename the solution $g^{\ast}(A,\Phi)$ of Lemma \ref{thm:asymptbehav} to $(A,\Phi)$. We refer to Eq.\ \eqref{eq:defspaceHk} below for the definition of the   space $H_{-1+\delta'}^2(r^{-1}\,dr)$ used in the next lemma.

\begin{lemma}\label{lem:normformlem}
Let $(A,\Phi)$ be a  solution of Eq.\ \eqref{eq:hitequ} as above and $\delta>0$ be the constant of Lemma \ref{thm:asymptbehav}. We fix a further constant  $0<\delta'<\min \{\frac{1}{2},\delta\}$. Then there exists    a complex gauge transformation $g=\exp(\gamma)\in \mathcal G^c$ where $\gamma\in H_{-1+\delta'}^2(r^{-1}\,dr)$    such that the restriction of $g^{\ast}(A,\Phi)$ to $\mathcal C_p(0)$  coincides with $(A_p^{\mod},\Phi_p^{\mod})$  for each $p\in\mathfrak p$.  
\end{lemma}

\begin{proof}
The assertion follows by combining the subsequent Propositions \ref{prop:diaggauge} and \ref{prop:gaugetomodel}.
\end{proof}

\begin{proposition}\label{prop:diaggauge}
On $\Sigma_0$ there exists a complex gauge transformation $g=\exp(\gamma)\in \mathcal G^c$ such that  $\gamma,r\nabla\gamma,r^2\nabla^2\gamma\in\mathcal O(r^{\delta})$ 
 with the following  significance.  On a sufficiently small noded subcylinder of    $\mathcal C_p(0)\subset\Sigma_0$, the pair $(A_1,\Phi_1):=g^{\ast}(A,\Phi)$ equals  
\begin{equation*}
\Phi_1=\Phi_p^{\mod}=\begin{pmatrix}C_p&0\\0&-C_p\end{pmatrix}\,\frac{dz}{z}
\end{equation*}
and
\begin{equation}\label{eq:connectionA_1}
A_1=A_p^{\mod}+\begin{pmatrix}\beta_p&0\\
0&-\beta_p\end{pmatrix}\left(\frac{dz}{z}-\frac{d\bar z}{z}\right)
\end{equation}
for some map $\beta_p$ satisfying $\beta_p,r\nabla \beta_p \in\mathcal O(r^{\delta})$.
\end{proposition}

\begin{proof}
By Lemma \ref{thm:asymptbehav}, the Higgs field $\Phi$ takes near any $p\in\mathfrak p$ the form
\begin{equation*}
\Phi=\begin{pmatrix}C_p+\varphi_0&\varphi_1\\\varphi_2&-C_p-\varphi_0\end{pmatrix}\,\frac{dz}{z}
\end{equation*}
for smooth functions $\varphi_j$ satisfying $\varphi_j,r\nabla\varphi_j,r^2\nabla^2\varphi_j\in\mathcal O(r^{\delta})$ for $j=0,1,2$. Because   $\det \Phi=-C_p^2\frac{dz^2}{z^2}$ by Eq.\ \eqref{eq:holQD}, there holds the relation
\begin{equation}\label{eq:detrelation}
2C_p\varphi_0+\varphi_0^2+\varphi_1\varphi_2=0.
\end{equation}
We claim the existence of   a complex gauge transformation $g$ which diagonalizes $\Phi$ and has the desired decay. To define it, we set for  $j=0,1,2$
\begin{equation*}
d_j:=- \frac{\varphi_j}{2C_p+\varphi_0} .
\end{equation*}
Then the complex gauge transformation
\begin{equation*}
g_p:=\frac{1}{\sqrt{1+d_0}} \begin{pmatrix} 1&d_1\\
d_2&1\end{pmatrix}
\end{equation*}   
has the  desired properties. Indeed, using Eq.\ \eqref{eq:detrelation}, it is easily seen that $g_p$ has determinant $1$ and satisfies $g_p^{-1}\Phi g_p=\Phi_p^{\mod}$. It remains to check that $g_p $     can be written as $g_p=\exp(\gamma_p)$, where $\gamma_p$ decays at   the   asserted   rate. As for the denominator $2C_p+\varphi_0$ appearing in the definition of $d_j$, it follows that for all $z$ sufficiently close to $0$ its modulus is  uniformly bounded away from $0$ since $\varphi_0\in \mathcal O(r^{\delta})$ and $C_p\neq 0$ by assumption $(A1)$. Hence each of the functions $d_j$,  $j=0,1,2$, satisfies $d_j,r\nabla d_j,r^2\nabla^2 d_j\in\mathcal O(r^{\delta})$.   The same decay properties then hold for  $g_p-\mathbbm{1}$ and thus there exits $\gamma_p$ as required.
\medskip\\
We extend the locally defined gauge transformations $g_p$ to a smooth gauge transformation $g$ on $\Sigma_0$. The gauge-transformed connection $A_1=g^{\ast}A$ is now automatically diagonal near each $p\in\mathfrak p$. Namely, by complex gauge-invariance of the second equation in Eq.\ \eqref{eq:hitequ} and the above    it satisfies    $\bar\partial_{A_1}\Phi_1=\bar\partial_{A_1}\Phi_p^{\mod}=0$. Hence
\begin{equation*}
0=\bar\partial\begin{pmatrix}C_p&0\\0&-C_p\end{pmatrix}\,\frac{dz}{z}+[A_1^{0,1}\wedge \Phi_p^{\mod}]=[A_1^{0,1}\wedge \Phi_p^{\mod}].
\end{equation*}
Thus  $A_1^{0,1}$ commutes with   $\Phi_p^{\mod}$ and therefore is diagonal. The same holds true for $A_1^{1,0}=-(A_1^{0,1})^{\ast}$, hence for $A_1$.  At $r=0$, the difference $B:=r(A_1-A)$ satisfies $B,r\nabla B  \in\mathcal O(r^{\delta})$ as follows from   the transformation formula
\begin{equation*}
A_1^{0,1}=g_p^{-1}A^{0,1}g_p+g_p^{-1}\bar\partial g_p=g_p^{-1}A^{0,1}g_p+\bar\partial\gamma_p
\end{equation*}
and the corresponding decay behavior of $\gamma_p$. If necessary, we apply a further unitary gauge transformation which can be chosen to be diagonal near each puncture $p\in\mathfrak p$, to put the connection $A_1$ into radial gauge. This last step leaves the Higgs field unchanged near  $p$. The asserted decay property of the connection then follows from a similar argument.  
\end{proof}

From now on, we again use the notation $\mathcal C_p(0)$ for  the noded subcylinder  occuring in the statement of the previous proposition. In view of this result, in order to show Lemma \ref{lem:normformlem} it remains to find a further  gauge transformation $g\in\mathcal G^c$ which fixes $\Phi_1$ and transforms the connection $A_1$ such that it coincides with  $A_p^{\mod}$ on each noded  cylinder  $\mathcal C_p(0)$. We look for $g$ in the form $g=\exp(\gamma)$ for some hermitian section $\gamma$ and first solve the easier problem of transforming via $g^{-1}$ the flat connection $A_p^{\mod}$ into a connection $A_2$ whose curvature coincides with that of $A_1$ on  $\mathcal C_p(0)$. We address  this problem  in Proposition  \ref{prop:gaugetomodel} below and assume for the moment the existence of such a section $\gamma$. That is, $\gamma^{\ast}=\gamma$ and $g=\exp(\gamma)$ satisfies 
\begin{equation*}
g^{-1}\Phi_1 g=\Phi_1,\qquad F_{(g^{-1})^{\ast}A_p^{\mod}}=F_{A_1}\qquad\textrm{locally on each}\;\mathcal C_p(0),\quad p\in\mathfrak p.
\end{equation*}
Put  $A_2:=g^{\ast}A_1$, which we in addition may assume to be in radial gauge, after modifying it,  if necessary, over each $\mathcal C_p(0)$ by a further diagonal and unitary gauge transformation of the same decay. Then writing
\begin{equation*}
A_2=\begin{pmatrix}\alpha_p+\beta_{2,p}&0\\0&-\alpha_p-\beta_{2,p}\end{pmatrix} \,d\theta\qquad\textrm{locally on each}\;\mathcal C_p(0),\quad p\in\mathfrak p,
\end{equation*}
for some function $\beta_{2,p}\in\mathcal O(r^{\delta})$ it follows that
\begin{equation*}
F_{A_2}=\begin{pmatrix}\partial_r\beta_{2,p}&0\\0&-\partial_r\beta_{2,p}\end{pmatrix}  \,dr\wedge d\theta=F_{A_p^{\mod}}=0.
\end{equation*}
Hence  $\beta_{2,p}\equiv0$ and $A_2=A_p^{\mod}$ locally on $\mathcal C_p(0)$, as desired.
\medskip\\
It remains to show the existence of the hermitian section $\gamma$ used in the above construction. It was shown in \cite[Proposition 4.4]{msww14} that this section $\gamma$ is determined as  solution  to the Poisson equation  
\begin{equation}\label{eq:Poisson}
\Delta_{A_p^{\mod}}\gamma=i\ast F_{A_1}^{\perp}
\end{equation}
for the connection Laplacian $\Delta_{A_p^{\mod}} \colon \Omega^0(i\mathfrak{su}(E))\to \Omega^0(i\mathfrak{su}(E))$, which we need to solve locally on each noded cylinder $\mathcal C_p(0)$. The Laplacian and Hodge-$\ast$ operator appearing here are taken with respect to the metric $g=\frac{\left|dz\right|^2}{\left|z\right|^2}$ on $\mathcal C_p(0)$ for reasons that will become clear later. We are    interested in finding a solution of Eq.\ \eqref{eq:Poisson} of the form
\begin{equation*}\label{eq:deffctu}
\gamma=\begin{pmatrix}u&0\\0&-u\end{pmatrix}
\end{equation*}
for some real-valued function $u$. With $\ast\, dr\wedge d\theta=r$ it follows from Eq.\ \eqref{eq:connectionA_1} that
\begin{equation*}
i\ast F_{A_1}^{\perp}=\begin{pmatrix}-2r\partial_r\beta_p&0\\0&2r\partial_r\beta_p\end{pmatrix}=:\begin{pmatrix}h_p&0\\0&-h_p\end{pmatrix}.
\end{equation*}
Therefore   Eq.\ \eqref{eq:Poisson} reduces to the  equation
\begin{equation}\label{eq:scalPoisson}
\Delta_0u=h_p
\end{equation}
for the scalar Laplacian $\Delta_0=-(r\partial_r)^2-\partial_{\theta}^2$. To obtain regularity estimates for solutions to  Eq.\ \eqref{eq:scalPoisson} it is convenient to   introduce the Hilbert spaces  
\begin{equation*}
L_{-1+\delta}^2(r^{-1}dr)=\left\{u\in L^2(\D)\mid  r^{-\delta}u\in L^2(r^{-1}dr)  \right\}
\end{equation*}
and 
\begin{multline}\label{eq:defspaceHk}
H_{-1+\delta}^k(r^{-1}dr) =\\
\left\{u\in L^2(\D)\mid (r\partial_r)^j\partial_{\theta}^{\ell} u\in L_{-1+\delta}^2(r^{-1}dr), 0\leq j+\ell\leq k\right\}.
\end{multline}
We note that $r\partial_r\beta_p\in\mathcal O(r^{\delta})$ as shown in Proposition \ref{prop:diaggauge}, and therefore  the right-hand side   of Eq.\ \eqref{eq:scalPoisson} satisfies $h_p\in\mathcal O(r^{\delta})$. From this it is straightforward to check that  $h_p\in L_{-1+\delta'}^2(r^{-1}dr)$ for any $\delta'<\delta$.

\begin{proposition}\label{prop:gaugetomodel}
Let $\delta>0$ be as before and fix a further constant $0<\delta'<\min \{\frac{1}{2},\delta\}$. Then the Poisson equation  $\Delta_0u=h$  on the punctured disk $\D^{\times} =\{z\in\C\mid 0<\left|z\right|<1\}$ admits a solution $u\in H_{-1+\delta'}^2(r^{-1}dr)$, which satisfies
\begin{equation*}
\|u\|_{H_{-1+\delta'}^2(r^{-1}dr)}\leq C \|h\|_{L_{-1+\delta}^2(r^{-1}dr)}
\end{equation*}
for some constant $C=C(\delta,\delta')$ which does not depend on $h\in L_{-1+\delta}^2(r^{-1}dr)$.
\end{proposition}

\begin{proof}
Fourier decomposition $u=\sum_{j\in\Z}u_je^{ij\theta}$ and $h=\sum_{j\in\Z}h_je^{ij\theta}$ reduces Eq.\ \eqref{eq:scalPoisson} to the system of ordinary differential equations
\begin{equation}\label{eq:Fourierode}
\left(-(r\partial_r)^2+j^2\right)u_j=h_j\qquad (j\in\Z)
\end{equation}
which we analyze in each Fourier mode   separately. For $j=0$ it has the solution
\begin{equation*}
u_0(r)=-\log r\int_0^rh_0(s)s^{-1}\,ds+\int_0^rh_0(s)s^{-1}\log s\,ds.
\end{equation*}
We estimate the two terms on the right-hand side separately. By the Cauchy-Schwarz inequality it follows that
\begin{eqnarray*}
\lefteqn{\int_0^1\left|\log r\int_0^rh_0(s)\,\frac{ds}{s} \right|^2r^{-2\delta'}\,\frac{dr}{r}}\\
&\leq&\int_0^1(\log r)^2\left(\int_0^r|h_0(s)|^2s^{-2\delta}\,\frac{ds}{s}\right) \left(\int_0^rs^{2\delta}\,\frac{ds}{s}\right)r^{-2\delta'}\,\frac{dr}{r}\\
&\leq&\int_0^1(\log r)^2 \left(\int_0^rs^{2\delta}\,\frac{ds}{s}\right)r^{-2\delta'}\,\frac{dr}{r} \cdot \int_0^1|h_0(s)|^2s^{-2\delta}\,\frac{ds}{s}\\
&=&\frac{1}{2\delta}\int_0^1(\log r)^2  r^{2(\delta-\delta')}\,\frac{dr}{r}\cdot \|h_0\|_{L_{-1+\delta}^2(r^{-1}dr)}^2\\
&\leq&C(\delta,\delta')\|h_0\|_{L_{-1+\delta}^2(r^{-1}dr)}^2.
\end{eqnarray*}
Similarly, using again the Cauchy-Schwarz inequality and for any constant $0<\delta''<\delta-\delta'$  the elementary inequality
\begin{equation*}
0\geq s^{\delta''}\log s\geq -\frac{1}{\delta''e}=: -C(\delta'')\qquad(0<s\leq1)
\end{equation*}
we obtain that 
\begin{eqnarray*}
\lefteqn{\int_0^1\left| \int_0^rh_0(s)\log s\,\frac{ds}{s} \right|^2r^{-2\delta'}\,\frac{dr}{r}}\\
&\leq&\int_0^1\left(\int_0^r\left|h_0(s)\right|^2s^{-2\delta}\,\frac{ds}{s}\right)\left(\int_0^r(\log s)^2s^{2\delta}\,\frac{ds}{s}\right)r^{-2\delta'}\,\frac{dr}{r}\\
&\leq&\|h_0\|_{L_{-1+\delta}^2(r^{-1}dr)}^2\cdot\int_0^1\left(\int_0^r(\log s)^2s^{2\delta}\,\frac{ds}{s}\right)r^{-2\delta'}\,\frac{dr}{r}\\
&\leq&\|h_0\|_{L_{-1+\delta}^2(r^{-1}dr)}^2\cdot C(\delta'')^2\int_0^1\left(\int_0^rs^{2(\delta-\delta'')}\,\frac{ds}{s}\right)r^{-2\delta'}\,\frac{dr}{r}\\
&\leq&\|h_0\|_{L_{-1+\delta}^2(r^{-1}dr)}^2\cdot C(\delta,\delta'')\int_0^1r^{2(\delta-\delta'-\delta'')}\,\frac{dr}{r}\\
&\leq& C(\delta,\delta',\delta'')\|h_0\|_{L_{-1+\delta}^2(r^{-1}dr)}^2.
\end{eqnarray*}
For $j\geq1$ a solution of Eq.\ \eqref{eq:Fourierode} is given by
\begin{eqnarray}\label{eq:integop}
\nonumber u_j(r)&=&\frac{r^{-j}}{2j}\int_0^rh_j(s)s^j\,\frac{ds}{s}-\frac{r^j}{2j}\int_1^rh_j(s)s^{-j}\,\frac{ds}{s}\\
&=&\frac{r^{-j}}{2j}\int_0^rh_j(s)s^js^{2\delta}\,\frac{ds}{s^{1+2\delta}}+\frac{r^j}{2j}\int_r^1h_j(s)s^{-j}s^{2\delta}\,\frac{ds}{s^{1+2\delta}}.
\end{eqnarray}
The integral kernel of the map $K_j\colon h_j\mapsto u_j=K_jh_j\colon L_{-1+\delta}^2(r^{-1}\,dr)\to L_{-1+\delta'}^2(r^{-1}\,dr)$ therefore is 
\begin{equation*}
K_j(r,s)=\begin{cases} \frac{1}{2j}r^{-j} s^{j+2\delta},&\textrm{if}\;0\leq s\leq r,\\
\frac{1}{2j}r^j s^{-j+2\delta},&\textrm{if}\;r\leq s\leq 1.\end{cases} 
\end{equation*}
In the case $j\leq-1$ a solution of Eq.\ \eqref{eq:Fourierode} is given by $u_j=K_{-j}h_j$. Applying Schur's test (cf.~\cite{hasu78}), we obtain for the operator  norm of  $K_j$, $j\geq1$, the bound
\begin{eqnarray*}
\lefteqn{\|K_j\|_{\mathcal L(L_{-1+\delta}^2(r^{-1}\,dr),L_{-1+\delta'}^2(r^{-1}\,dr))}\leq}\\
&&\sup_{0\leq r\leq 1}\int_0^1\left|K_j(r,s)\right|s^{-1-2\delta}\,ds+\sup_{0\leq s\leq 1}\int_0^1\left|K_j(r,s)\right|r^{-1-2\delta'}\,dr\\
&=&\sup_{0\leq r\leq 1}\left(\frac{1}{2j}\int_0^r\left|r^{-j} s^{j} s^{2\delta }\right|s^{-1-2\delta}\,ds +\frac{1}{2j}\int_r^1\left|r^j s^{-j}s^{2\delta }\right|s^{-1-2\delta}\,ds\right)\\
&&+\sup_{0\leq s\leq 1}\left( \frac{1}{2j}\int_s^1\left|r^{-j} s^{j} s^{2\delta}\right|r^{-1-2\delta'}\,dr +\frac{1}{2j}\int_0^s\left|r^j s^{-j}s^{2\delta}\right|r^{-1-2\delta'}\,dr  \right)\\
&=&\sup_{0\leq r\leq 1}\frac{1}{2j^2}(1+1-r^j)\\
&&+\sup_{0\leq s\leq 1}\left(\frac{1}{2j(j+2\delta')}(s^{2(\delta-\delta')}-s^{j+2\delta})  +\frac{1}{2j(j-2\delta')}s^{2(\delta-\delta')}\right)\leq \frac{4}{j^2}.
\end{eqnarray*}
The same bound  holds for the operator norm of $K_j$ in the case  $j\leq-1$.  Summing these estimates over $j\in\Z$ implies a bound for $u$ in $L_{-1+\delta'}^2(r^{-1}\,dr)$. The asserted $H_{-1+\delta'}^2(r^{-1}\,dr)$ estimate follows in a similar way by Fourier decomposition of the maps $r\partial_ru=\sum_{j\in\Z}r\partial_ru_j$ and  $(r\partial_r)^2u=\sum_{j\in\Z}(r\partial_r)^2u_j$. In the case $j\neq0$, the resulting terms are up to multiples of $j$, respectively of $j^2$, similar to those in Eq.\ \eqref{eq:integop}. Using again Schur's test it follows that the norms of the corresponding integral kernels are bounded by $\frac{C}{j}$, respectively by $C$, for some constant $C$ which does not depend on $j$, and thus can be summed over. In the case $j=0$, it remains to check the asserted estimate for $r\partial_r u_0$ and $(r\partial_r)^2u_0$. Using Eq.\ \eqref{eq:Fourierode}, the latter equals $-h_0$ and thus is bounded by $C\|h_0\|_{L_{-1+\delta}^2(r^{-1}dr)}$. As for
\begin{equation*}
r\partial_r u_0=-\int_0^rh_0(s)\frac{ds}{s},
\end{equation*}
we use the Cauchy-Schwarz inequality to estimate
\begin{eqnarray*}
\lefteqn{\int_0^1\left| \int_0^rh_0(s) \,\frac{ds}{s} \right|^2r^{-2\delta'}\,\frac{dr}{r}}\\
&\leq&\int_0^1\left(\int_0^r\left|h_0(s)\right|^2s^{-2\delta}\,\frac{ds}{s}\right)\left(\int_0^r s^{2\delta}\,\frac{ds}{s}\right)r^{-2\delta'}\,\frac{dr}{r}\\
&\leq&\|h_0\|_{L_{-1+\delta}^2(r^{-1}dr)}^2\cdot\int_0^1\left(\int_0^r s^{2\delta}\,\frac{ds}{s}\right)r^{-2\delta'}\,\frac{dr}{r}\\
&\leq& C(\delta,\delta')\|h_0\|_{L_{-1+\delta}^2(r^{-1}dr)}^2.
\end{eqnarray*}
This completes the proof.
\end{proof}

After these preparations we can now define  suitable  approximate   solutions to   Eq.\ \eqref{eq:hitequ}. Let  $(A,\Phi)$ be the   exact solution on the noded surface $\Sigma_0$  as  in the beginning of \textsection \ref{subsect:approxsolutions},  and let $\delta'>0$ be the constant of Lemma \ref{lem:normformlem}.   This lemma shows the existence of some $\gamma\in H_{-1+\delta'}^2(r^{-1}dr)$   such that $\exp(\gamma)^{\ast}(A,\Phi)$ coincides with $(A_p^{\mod},\Phi_p^{\mod})$ on each punctured cylinder $\mathcal C_p(0)$.  Let $r>0$ be a smooth function on $\Sigma_0$ which coincides with $\left|z\right|$, respectively $\left|w\right|$ near each puncture.  Fix a constant $0<R<1$ and a smooth cutoff function $\chi_R\colon[0,\infty)\to[0,1]$ with support in $[0,R]$ and such that $\chi_R(r)=1$ if $r\leq\frac{3R}{4}$. We impose the further requirement that
\begin{equation}\label{eq:assumptcutoff}
\left|r\partial_r\chi_R\right|+ \left|(r\partial_r)^2\chi_R\right|\leq C
\end{equation}
for some constant $C$ which does not depend on $R$. Then the map $p\mapsto \chi_R(r(p))\colon \Sigma_0\to\R$ gives rise to a smooth cutoff function on $\Sigma_0$ which by a slight abuse of notation we again denote $\chi_R$. Then we define  
\begin{equation}\label{eq:apprsol}
(A_R^{\app},\Phi_R^{\app}):=\exp(\chi_R\gamma)^{\ast}(A,\Phi).
\end{equation}
Note that by choice of the section $\gamma$, the pair $(A_R^{\app},\Phi_R^{\app})$ is an exact solution of   Eq.\ \eqref{eq:hitequ}   in the region
\begin{equation*}
\Sigma_0\setminus\bigcup_{p\in\mathfrak p}\Big\{z\in\mathcal C_p(0)\mid \frac{3R}{4}\leq\left|z\right|\leq R\Big\}.
\end{equation*}
It equals to $(A_p^{\mod},\Phi_p^{\mod})$ on the subset of points in $\mathcal C_p(0)$ which satisfy $0<\left|z\right|\leq \frac{3R}{4}$.
Set $\rho=\frac{R}{2}$ and let $t\in\C$ with $\left|t\right|=\rho^2$. Then let $\Sigma_{t}$ be the smooth Riemann surface obtained from $\Sigma_0$ by identifying  $zw=t$ at each node $p\in\mathfrak p$, cf.~\textsection\ref{subsect:degeneratingfamilies}. By  construction of the model solutions $(A_p^{\mod},\Phi_p^{\mod})$ it follows that   the restriction of  $(A_R^{\app},\Phi_R^{\app})$ to the subdomain
\begin{equation*}
\Sigma_0\setminus \bigcup_{p\in\mathfrak p}\Big\{z\in\mathcal C_p(0)\mid  \left|z\right|\leq \frac{R}{2}\Big\}
\end{equation*}
extends smoothly over the cut-locus $\left|z\right|= \left|w\right|= \frac{R}{2}=\rho$ and hence defines a   smooth pair  $(A_R^{\app},\Phi_R^{\app})$   on  $\Sigma_{t}$. We call $(A_R^{\app},\Phi_R^{\app})$ the {\bf{approximate solution to the parameter $R$}}. By definition, it satisfies the second equation in Eq.\ \eqref{eq:hitequ} exactly, and the first equation up to some error which we have good control on.

\begin{figure}   
\label{bild2}                                                   
\includegraphics[width=1.1\textwidth]{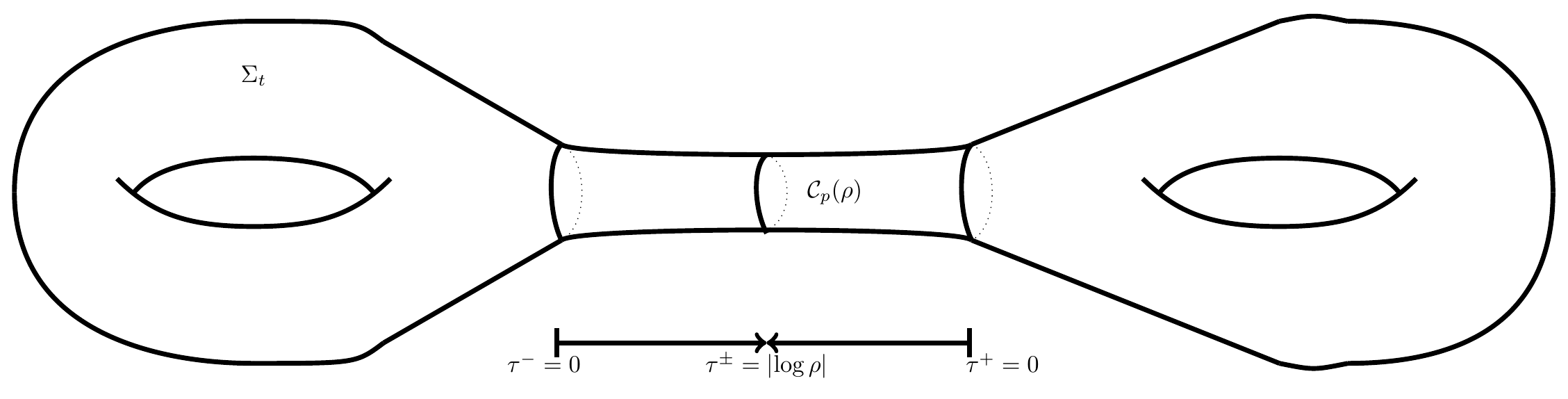}                        \caption{Setup for the definition of the approximate solution $(A_R^{\app},\Phi_R^{\app})$ on a  smooth Riemann surface $\Sigma_t$ with a long cylindrical ``neck". It interpolates between the given solution $(A_0,\Phi_0)$ on the ``thick" part of $\Sigma_t$ and the  model solution $(A_p^{\mod},\Phi_p^{\mod})$ on the  ``neck".}
\end{figure}

\begin{lemma}\label{lem:errorappr}
Let $\delta'>0$ be as above, and fix some further constant  $0<\delta''<\delta'$.  The approximate solution $(A_R^{\app},\Phi_R^{\app})$ to the parameter $0<R<1$ satisfies
\begin{equation*}
\|\ast F_{A_R^{\app}}^{\perp}+\ast [\Phi_R^{\app}\wedge (\Phi_R^{\app})^{\ast}]\|_{C^0(\Sigma_t)}\leq CR^{\delta''}
\end{equation*}
for some constant $C=C(\delta',\delta'')$ which does not depend on  $R$.  
\end{lemma}

\begin{proof}
It suffices to estimate the error in the union over $p\in\mathfrak p$ of the regions  $S_p:= \{z\in\mathcal C_p(0)\mid \frac{3R}{4}\leq\left|z\right|\leq R\}$ since $(A_R^{\app},\Phi_R^{\app})$ is an exact solution on its complement. For each $p\in\mathfrak p$ we obtain, using the triangle inequality,
\begin{multline*}
\|\ast F_{A_R^{\app}}^{\perp}+\ast [\Phi_R^{\app}\wedge (\Phi_R^{\app})^{\ast}]\|_{C^0(S_p)}\\
\leq \|\ast (F_{A_R^{\app}}^{\perp}- F_A^{\perp})\|_{C^0(S_p)}+ \|\ast ( [\Phi_R^{\app}\wedge (\Phi_R^{\app})^{\ast}]- [\Phi\wedge\Phi^{\ast}])\|_{C^0(S_p)}
\end{multline*}
since by assumption $(A,\Phi)$ is a solution of Eq.\ \eqref{eq:hitequ} on $\Sigma_0$. With $g=\exp(\chi_R\gamma)$ as in Eq.\ \eqref{eq:apprsol} it follows that
\begin{equation*}
\ast F_{A_R^{\app}}^{\perp}=\ast F_{g^{\ast}A}^{\perp}=g^{-1}(\ast F_A^{\perp}+i\Delta_A(\chi_R\gamma) )g,
\end{equation*}
and hence
\begin{multline*}
\|\ast F_{A_R^{\app}}^{\perp}-\ast F_A^{\perp}\|_{C^0(S_p)}\leq \|\ast g^{-1} F_A^{\perp}g-\ast F_A^{\perp}\|_{C^0(S_p)}+\|g^{-1}\Delta_A(\chi_R\gamma) g\|_{C^0(S_p)}.
\end{multline*}
We estimate the two terms on the right-hand side separately. Developing $g=\exp(\chi_R\gamma)$ in an exponential series, we obtain  for  the first one  
\begin{equation*}
\|\ast g^{-1} F_A^{\perp}g-\ast F_A^{\perp}\|_{C^0(S_p)}\leq C\|\gamma\|_{C^0(S_p)}\leq CR^{\delta''}
\end{equation*}
for some constant $C$ independent of $R$. In the last step we used that by Lemma \ref{lem:normformlem} $\gamma\in H_{-1+\delta'}^2(r^{-1}\,dr)$ together with the continuous embeddings
\begin{equation*}
H_{-1+\delta'}^2(S_p(r^{-1}\,dr))\hookrightarrow H_{-1+\delta''}^2(S_p(r^{-1}\,dr)) \hookrightarrow C^0(S_p)
\end{equation*}
for exponents $0<\delta''<\delta'$. Note that the norm of first embedding is bounded above by $CR^{\delta'-\delta''}$. The decay of the second term follows from
\begin{equation}\label{eq:Delta_Agamma}
\Delta_A(\chi_R\gamma)=\chi_R\Delta_A\gamma+2r\partial_r\chi_R\cdot r\partial_r\gamma+(r\partial_r)^2\chi_R\cdot\gamma,
\end{equation}
where we recall that in some neighborhood of $p$ the term $\Delta_A\gamma$ equals the right-hand side in  Eq.\ \eqref{eq:Poisson}. As discussed before Proposition \ref{prop:gaugetomodel}, it has order of decay $\mathcal O(r^{\delta})$. The last two summands in Eq.\ \eqref{eq:Delta_Agamma} decay as $\mathcal O(r^{\delta''})$  as follows from Proposition \ref{prop:gaugetomodel} and the assumption in Eq.\ \eqref{eq:assumptcutoff} on $\chi_R$.  It remains to show the  estimate  
\begin{equation*}
\|  \ast ([g^{-1}\Phi g\wedge(g^{-1}\Phi g)^{\ast}]-[\Phi\wedge\Phi^{\ast}])\|_{C^0(S_p)}\leq CR^{\delta''},
\end{equation*}
which as before follows from $\|g-\operatorname{id}\|_{C^0(S_p)}\leq CR^{\delta''}$. 
\end{proof}

%
\subsection{Linearization of the Hitchin operator along a complex gauge orbit}\label{lin.hit.ope}

With the aim of finally perturbing the above defined approximate solution $(A_R^{\app},\Phi_R^{\app})$ to an exact solution in mind, we first need to study the linearization of the first equation in Eq.\ \eqref{eq:hitequ} in some detail.
\medskip\\
We fix a sufficiently small  constant $R>0$ and refer to \textsection \ref{subsect:approxsolutions} for the definition of the closed surface $\Sigma_t$, where $\left|t\right|=\rho=\frac{R}{2}$. On $\Sigma_t$ we study the nonlinear Hitchin operator  
\begin{align}\label{eq:nonlinH}
\nonumber\mc H  \colon H^1(\Sigma_t,\Lambda^1\otimes \mf{su}(E)\oplus\Lambda^{1,0}\otimes \mf{sl}(E)) \to L^2(\Sigma_t,\Lambda^2\otimes  \mf{su}(E)
\oplus\Lambda^{1,1}\otimes\mf{sl}(E)),\\
\mc H(A,\Phi) = ( F^\perp_A+  [\Phi \wedge \Phi^*],\bar\partial_A\Phi).
\end{align}
We further consider the orbit map for the action of the group $\mathcal G^c$ of complex gauge transformations,
\begin{equation}\label{orb.map}
\gamma\mapsto \mathcal{O}_{(A,\Phi)}(\gamma)=g^{\ast}(A,\Phi)=( g^{\ast}A,g^{-1}\Phi g),  
\end{equation}
where $\gamma\in H^1(\Sigma_t,i \mf{su}(E))$ and $g=\exp(\gamma)$.  Given a Higgs pair $(A,\Phi)$, i.e.~a solution of $\bar\partial_A\Phi=0$, our goal   is to find a point  in the complex gauge orbit of it at which $\mc H$ vanishes. Since the condition that $\bar{\partial}_A \Phi = 0$ is invariant under  complex gauge transformations, we in fact only need to find a solution $\gamma$ of 
\begin{equation}\label{eq:defmapFt}
\mc{F}(\gamma) := \pr_1 \circ \mathcal{H} \circ \mathcal{O}_{(A,\Phi)}( \gamma) = 0,
\end{equation}
or equivalently, of
\begin{equation*}
F_{g^{\ast}A}^\perp+ [g^{-1} \Phi g\wedge(g^{-1}\Phi g)^{\ast}] = 0, \qquad\textrm{where}\; g = \exp(\gamma).
\end{equation*}
By continuity of the multiplication maps $H^1\cdot H^1\to L^2$ and $H^2\cdot H^1\to H^1$, it is easily seen that the map $\mathcal H$ in Eq.\ \eqref{eq:nonlinH} and the  maps
\begin{align*}
\begin{split}
&\mathcal{O}_{(A,\Phi)}\colon H^2(\Sigma_t,i\mf{su}(E)) \to  H^1(\Sigma_t,\Lambda^1\otimes\mf{su}(E)\oplus\Lambda^{1,0}\otimes\mf{sl}(E)), \label{eq:nonlinmaps}\\ 
&\mc{F}\colon H^2(\Sigma_t,i\mf{su}(E)) \to L^2(\Sigma_t,\Lambda^2\otimes \mf{su}(E)) 
\end{split}
\end{align*}
are all well-defined and smooth. We now compute their linearizations.  First, the differential at $g =\operatorname{id}$ of the orbit map in Eq.\ \eqref{orb.map} is 
\begin{equation*} 
\Lambda_{(A,\Phi)}\gamma=(\bar\partial_A\gamma-\partial_A\gamma,[\Phi\wedge \gamma]).
\end{equation*}
Furthermore, the differential of the Hitchin operator at $(A,\Phi)$ is
\begin{equation}\label{eq:diffHitchinop}
D\mc H \begin{pmatrix} \alpha \\ \varphi \end{pmatrix}= \begin{pmatrix}
 d_A &   [\Phi \wedge {\cdot\,}^{\ast}]+ [\Phi^{\ast}\wedge\cdot\,]\\[0.5ex]
[\Phi\wedge\cdot\,] & \bar{\partial}_A
 \end{pmatrix}\begin{pmatrix} \alpha \\ \varphi \end{pmatrix},
\end{equation}
whence
\begin{equation*} 
(D\mc H \circ \Lambda_{(A,\Phi)})(\gamma) = \begin{pmatrix}(\partial_A\bar\partial_A-\bar\partial_A\partial_A)
\gamma+ [\Phi\wedge[\Phi\wedge\gamma]^{\ast}]+[\Phi^{\ast}\wedge[\Phi\wedge\gamma]]\\[0.5ex] 
[\Phi\wedge(\bar\partial_A\gamma-\partial_A\gamma)]+\bar\partial_A[\Phi\wedge\gamma]
\end{pmatrix}.
\end{equation*}
Using that $\bar\partial_A\Phi=0$, as well as the fact that
$[\Phi\wedge\partial_A\gamma]=0$ for dimensional reasons, the entire second component vanishes. The first component equals $D\mc{F}(\gamma) =i\ast L_{(A ,\Phi )}$, with $L_{(A ,\Phi )}$  the elliptic operator   
\begin{equation}\label{eq:defnlinoperator}
L_{(A ,\Phi )}=\Delta_{A }-i \ast M_{\Phi }\colon \Omega^0(\Sigma_t,i\mathfrak{su}(E))\to \Omega^0(\Sigma_t,i\mathfrak{su}(E)).
\end{equation}
Here we define
\begin{equation}\label{eq:defnlinoperatorMphi}
M_\Phi\gamma:=[\Phi^{\ast}\wedge[\Phi\wedge\gamma]] - [ \Phi \wedge [ \Phi^{\ast}\wedge\gamma]].
\end{equation} 
Observe that the linear operators
\begin{align*}
\Lambda_{(A,\Phi)}\colon\Omega^0(\Sigma_t,i\mf{su}(E))\to\Omega^1(\Sigma_t,\mf{su}(E))\oplus\Omega^{1,0}(\Sigma_t,\mf{sl}(E)), \\
D\mc{F}=\pr_1\circ D\mc H \circ\Lambda_{(A,\Phi)}\colon\Omega^0(\Sigma_t,i\mf{su}(E))\to\Omega^2(\Sigma_t,\mf{su}(E)), \\
\mbox{and} \qquad
L_{(A ,\Phi )} \colon\Omega^0(\Sigma_t,i \su(E)) \rightarrow \Omega^0(\Sigma_t,i \su(E))
\end{align*}
are all bounded from $H^1$ to $L^2$, or $H^2$ to $L^2$ respectively. It is a basic fact, first observed by Simpson  \cite{si88}, that the linearized operator $L_{(A ,\Phi )}$ is nonnegative.

\begin{lemma}\label{positivity}
If $\gamma \in \Omega^0(\Sigma_t,i \su(E))$, then
\begin{equation*} 
\langle   L_{(A ,\Phi )} \gamma, \gamma \rangle_{L^2(\Sigma_t)} =  \|d_A \gamma\|_{L^2(\Sigma_t)}^2 + 2 \| [ \Phi\wedge\gamma ]\|_{L^2(\Sigma_t)}^2 \geq 0.
\end{equation*} 
In particular, $L_{(A ,\Phi )} \gamma =0$ if and only if $d_A\gamma=[\Phi \wedge\gamma ] = 0$.
\end{lemma}

\begin{proof}
For a proof we refer to \cite[Proposition 5.1]{msww14}.
\end{proof}

In the case of  rank-$2$ Higgs bundles which we are considering  here, we can characterize the nullspace of the algebraically defined operator $M_\Phi$ more closely. We fix $\varphi\in\mf{sl}(2,\C)$ and consider the linear map 
\begin{equation*}
M_\varphi\colon i\su(2) \to i\su(2),\quad\gamma \mapsto [ \varphi^*, [ \varphi, \gamma]] + [\varphi, [ \varphi^*, \gamma]].
\end{equation*}
By a short calculation, 
\begin{equation}\label{emmvarvieh}
\langle M_\varphi \gamma, \gamma \rangle =  \left|[\varphi,\gamma] \right|^2 +   \left|[\varphi^*,\gamma] \right|^2 = 2 \left|[\varphi,\gamma] \right|^2.
\end{equation}
Clearly $M_\varphi$ is hermitian with respect to $\langle \cdot \,, \cdot \rangle$ and satisfies 
$g^{-1}(M_\varphi\gamma)g=M_{g^{-1}\varphi g}g^{-1}\gamma g$ when $g\in \SU(2)$.

\begin{lemma}\label{lem:kernelMphi}
If $\varphi \in \mf{sl}(2,\C)$, then $M_\varphi\colon i\su(2) \to i\su(2)$ is invertible if and only if $[\varphi, \varphi^*] \neq 0$, i.e.\ if the endomorphism $\varphi$ is not normal. 
If $[\varphi,\varphi^*]=0$ for some $0 \neq \varphi \in \mf{sl}(2,\C)$, then $M_\varphi$ has a one-dimensional kernel.
\end{lemma}

\begin{proof}
Let $0\neq \varphi \in \mf{sl}(2,\C)$ be given and suppose that $0\neq \gamma\in \ker M_\varphi$. Then Eq.\ \eqref{emmvarvieh} shows that $[\varphi,\gamma]=0$. Since $\gamma^*=\gamma$ it follows that $0=[\varphi,\gamma]^*=-[\varphi^*,\gamma]$. Because $\gamma\neq0$, the identities $[\varphi,\gamma]= [\varphi^*,\gamma]=0$ can only be satisfied if $\varphi^*=\lambda \varphi$ for some $\lambda\in\C$. But then $[\varphi^*,\varphi]=0$, and $\varphi$ is normal. Conversely, suppose that $\varphi\neq0$ is normal. Then $\ker M_\varphi$ is at most one-dimensional. 
Consider $\gamma=\varphi+\varphi^*\in i\su(2)$. It satisfies $[\varphi,\gamma]=0$ by normality of $\varphi$, hence is contained in $\ker M_\varphi$. In the case where $\varphi+\varphi^*\neq0$ this shows that $\ker M_\varphi$ is exactly one-dimensional. If $\varphi+\varphi^*=0$ we can draw the same conclusion using $0\neq\gamma=i\varphi\in i\su(2)$.
\end{proof}

%
%
\subsection{Cylindrical Dirac-type operators}\label{subsect:cyloperators}

In order to carry out our intended gluing construction we need to have good control on the linearized operator $L_R:=L_{(A_R^{\app},\Phi_R^{\app})}$ in Eq.\ \eqref{eq:defnlinoperator}, and in particular we need a bound for the norm of its inverse $G_R:=L_R^{-1}$ as an operator acting on $L^2(\Sigma_t)$. As it turns out, the operator $L_R$ is strictly positive, and thus the operator norm of $G_R$ is bounded by $\lambda_R^{-1}$, the inverse of the smallest eigenvalue of $L_R$. Recall from \textsection \ref{subsect:degeneratingfamilies} that the Riemann surface $\Sigma_t$ is the disjoint union of a ``thick" region and a number of  ``long" euclidean cylinders $\mathcal C_p(R)$ of length $2\left|\log R\right|$, where $p\in\mathfrak p$.  We therefore expect that $\lambda_R^{-1}\sim \left|\log R\right|^2=:T^2$, i.e.~that this quantity diverges as $R\searrow 0$.
\medskip\\
To show that the   scale of $\lambda_R^{-1}$ takes place as expected, we shall make use of the Cappell--Lee--Miller  gluing theorem (cf.~\cite{clm96}) and its generalization to  small perturbations of constant coefficient operators due to Nicolaescu (cf.~\cite{ni02}). 
The general setup considered there is that of  a family of  manifolds $M_T$, $T_0\leq T\leq\infty$, each containing a long cylindrical ``neck" of length $\sim T= \left|\log R\right|$, thought of as being glued from two disjoint manifolds $M_T^{\pm}$ along the boundaries of a pair of cylindrical ends. We then consider  a   self-adjoint   first-order  Dirac-type operator    $\mathfrak  D_T$ on a hermitian  vector bundle over  $M_T$ (here ``Dirac-type"  means  that the square of $\mathfrak  D_T$ is a generalized Laplacian).  The statement of the Cappell--Lee--Miller  gluing theorem  is that under suitable assumptions   to which we return below, $\mathfrak  D_T$ admits two types of eigenvalues, referred to as \emph{large} and \emph{small} eigenvalues. Large eigenvalues are those of   order  of decay $O(T^{-1})$, while small ones are those decaying as $o(T^{-1})$. For $T\to\infty$ the subspace of $L^2$ spanned by the eigenvectors to small eigenvalues is  parametrized in a sense we do not make precise here   by the  kernel of the limiting operator $\mathfrak  D=\mathfrak  D_{\infty}$. Hence the occurence of small eigenvalues can be viewed as caused by the fact that the dimension of the kernel of $\mathfrak  D_T$ is in general  unstable as $T$ varies.
\medskip\\ 
Of particular interest to us   is a version of the generalized  Cappell--Lee--Miller  gluing theorem for a $\Z_2$-graded Dirac-type operator $\mathfrak D$ of the form considered in Eq.\ \eqref{eq:diracoperator} below. It acts on sections of the $\Z_2$-graded hermitian vector bundle $\hat E=\hat E^+\oplus \hat E^-$. Under the assumption that the operator $\slashed D\colon C^{\infty}(\hat E^+)\to C^{\infty}(\hat E^-)$ has trivial kernel one can show that for sufficiently large values of the gluing parameter $T$ the associated Laplacian $\slashed D_T^{\ast}\slashed D_T$ does not have any small eigenvalues at all and hence admits a bounded inverse of norm $\sim T^2$.  This is the content of Theorem \ref{thm:applicationclm} below. An application of it will almost immediately give us the desired control on the norm of the  operator $G_R$.
\medskip\\ 
Let us therefore digress here to introduce the setup for  Theorem     \ref{thm:applicationclm} and then  to explain how it is applied in the present context. Following closely \cite[\textsection 1--2]{ni02}, we consider an oriented Riemannian manifold $(\hat N,\hat g)$ with a cylindrical end modeled by $\R^+\times N$, where    $(N,g)$ is an oriented compact Riemannian manifold. We let $\pi\colon \R^+\times N\to N$ denote the canonical projection and introduce $\tau$ as the outgoing longitudinal coordinate on $\R^+\times N$. We further assume that $\hat g=d\tau^2\oplus g$ along the cylindrical end. A vector bundle $\hat E$ over $\hat N$ is called \emph{cylindrical}  if there exists a vector bundle $E\to N$ and a bundle isomorphism $\hat\vartheta\colon \hat E|_{\R^+\times N}\to \pi^{\ast}E$. A section $\hat u$ of $\hat E$ is called \emph{cylindrical} if there exists a section $u$ of $E$ such that along the cylindrical end, $\hat\vartheta \hat u=\pi^{\ast} u$. In this case we shall simply write $\hat u=\pi^{\ast} u$ and $u=\partial_{\infty}\hat u$. A cylindrical vector bundle comes equipped with a canonical first order partial differential operator $\partial_{\tau}$ acting on sections of $  \hat E|_{\R^+\times N}$. It is uniquely defined by the properties that $\partial_{\tau}(f  v)=\frac{\partial f}{\partial \tau}  v+f\partial_{\tau}v$ for every section $v$ of $\hat E|_{\R^+\times N}$ and  function $f\in C^{\infty}(\R^+\times N)$, and $\partial_{\tau}\hat u=0$  for every cylindrical section $\hat u$. 
\medskip\\
The vector bundles to be considered here are  $\Z_2$-graded   cylindrical hermitian vector bundles $(\hat E,\hat H)$. By definition, this means that $\hat E$  is a  cylindrical vector bundle,     the hermitian metric $\hat H$ is along the cylindrical end of the form  $\hat H=\pi^{\ast}H$   for some hermitian metric $H$ on $E$, and $\hat E$  splits into the orthogonal sum $\hat E= \hat E^+\oplus \hat E^-$ of cylindrical vector bundles. We also suppose that $\hat E$ carries a Clifford structure and let $G\colon E^+\to E^-$ denote the bundle isomorphism given by Clifford multiplication by $d\tau$. 
\medskip\\
Let $\hat E \to \hat N$ be a  $\Z_2$-graded cylindrical hermitian vector bundle. A  first order partial differential operator     $\mathfrak D\colon C^{\infty}(\hat E)\to C^{\infty}(\hat E)$   is called a \emph{$\Z_2$-graded cylindrical Dirac-type operator} if with respect to the $\Z_2$-grading of $\hat E$   it takes the form
\begin{equation}\label{eq:diracoperator}
\mathfrak D=\begin{pmatrix}
0& 	\slashed  D^{\ast}\\ \slashed D&0
\end{pmatrix},
\end{equation}
such that along the cylindrical end $\slashed D=G(\partial_\tau-D)$ for a self-adjoint Dirac-type operator $D\colon C^{\infty}(E^+)\to C^{\infty}( E^+) $. More generally, we need to consider the perturbed operator $\mathfrak D+ \mathfrak B$, where $\slashed D$ is replaced by the operator $\slashed D+ B$ (and $\slashed  D^{\ast}$ by $\slashed  D^{\ast}+B^{\ast}$). Here we assume that the perturbation $B$ is an exponentially decaying operator of order $0$, i.e.\ that there exist constants $C,\lambda>0$ such that
\begin{equation}\label{eq:expdecaycond}
\sup\{ | B(x)|\mid x\in [\tau,\tau+1]\times N \}\leq Ce^{-\lambda |\tau|}
\end{equation}
for all $\tau\in \R^+$.
\medskip\\
Now let a pair $(\hat N_i,\hat g_i)$, $i=1,2$, of oriented Riemannian manifolds with cylindrical ends be given. We endow the cylinder $(-\infty,0)\times N_2$ with the outgoing coordinate $-\tau<0$ and impose the following compatibility assumptions. First we assume that   there exists an orientation reversing isometry $\varphi\colon (  N_1,  g_1)\to ( N_2, g_2)$.  Let then   $\hat E_i\to \hat N_i$  be a pair of  $\Z_2$-graded cylindrical hermitian vector bundles such that there exists an isometry $\gamma\colon E_1\to E_2$ of hermitian vector bundles     covering $\varphi$ and respecting the gradings. Further, let $\mathfrak D_i$, $i=1,2$, be $\Z_2$-graded cylindrical Dirac-type operators as in Eq.\ \eqref{eq:diracoperator}. We assume that  the operators  $\slashed D_i$ appearing  there are of the form  $\slashed D_i=G_i(\partial_\tau-D_i)$ along the cylindrical ends, and that (with respect to the  identification  of the vector bundles along the ends induced by   $\gamma$)  $G_1+G_2=L_1+L_2=0$. We can then form   for each $T>0$ the smooth manifold $N_T$ obtained by attaching $\hat N_1\setminus  (T+1,\infty)\times N_1$ to $\hat N_2\setminus  (-\infty,-T-1)\times N_2$ using the the orientation preserving identification
\begin{equation*}
	[T+1,T+2]\times N_1\to [-T-2,-T-1]\times N_2,\quad (t,x)\mapsto (t-2T-3,\varphi(x)).
\end{equation*}
The   $\Z_2$-graded cylindrical hermitian  vector bundles $\hat E_i$ are glued in a similar way to obtain a  $\Z_2$-graded hermitian vector bundle $  E_T= E_T^+\oplus   E_T^-$ over $N_T$. We also note that   the  cylindrical operators $\mathfrak D_i$ combine to give a   $\Z_2$-graded Dirac-type operator $\mathfrak D_T$ on the  vector bundle $E_T$. We shall again allow for a perturbed version of the operator  $\mathfrak D_T$. Thus  let  $\mathfrak B_i$, $i=1,2$, be two perturbations satisfying the exponential decay condition in Eq.\ \eqref{eq:expdecaycond}. We fix a smooth cutoff function $\chi\colon\R\to [0,1]$ with support in $(-\infty,\frac{3}{4}]$ and such that $\chi(\tau)\equiv 1$ if $\tau\leq  \frac{1}{4}$, and set $\chi_T(\tau):=\chi(|\tau|-T)$. As before, the two perturbed operators $\mathfrak D_i+\chi_T\mathfrak B_i$   combine into a Dirac-type operator on $E_T$. From now on, the notation  $\mathfrak D_T$ will always refer to this perturbed Dirac-type operator, which we write as
\begin{equation*}
\mathfrak D_{T}=\begin{pmatrix}
0& 	\slashed  D_T^{\ast}\\ \slashed D_T&0
\end{pmatrix}.
\end{equation*}
We also introduce the notation $\mathfrak D_{i,\infty}:=\mathfrak D_i+ \mathfrak B_i$, $i=1,2$, and write
\begin{equation}\label{eq:diracoperatorpert}
\mathfrak D_{i,\infty}=\begin{pmatrix}
0& 	\slashed  D_{i,\infty}^{\ast}\\ \slashed D_{i,\infty}&0
\end{pmatrix}.
\end{equation}
We  finally  define the \emph{extended $L^2$ space} $L_{\operatorname{ext}}^2(\hat N,\hat E)$ to consist of all sections $\hat u$ of $\hat E$   such that  there exists an $L^2$ section  $u_{\infty}$ of $E$ satisfying
\begin{equation*}
\hat u - \pi^{\ast}  u_{\infty}\in L^2(\hat N,\hat E).
\end{equation*}
We also note that $ u_{\infty}$ is uniquely determined by $\hat u$ and so there is a well-defined map
\begin{equation*}
\partial_{\infty}\colon L_{\operatorname{ext}}^2(\hat N,\hat E)\to L^2(N,E),\quad \hat u\mapsto u_{\infty},
\end{equation*}
called \emph{asymptotic trace map}. The following theorem is  proved in \cite{ni02} as a particular instance of the generalized Cappell--Lee--Miller  gluing theorem.

\begin{theorem}\label{thm:applicationclm}
For $i=1,2$, let $\mathfrak D_{i,\infty}$ be a  $\Z_2$-graded   Dirac-type operator 	 on the cylindrical vector bundle $\hat E_i\to \hat N_i$ as in Eq.\ \eqref{eq:diracoperatorpert}. Suppose that the kernel $K_i^+\subseteq L_{\operatorname{ext}}^2(\hat N_i,\hat E_i^+)$ of the operator $\slashed D_{i,\infty}$ is trivial for $i=1,2$. Then there is $T_0>0$ and a constant $C>0$  such that the operator $\slashed D_T^{\ast} \slashed D_T$ is bijective for all $T>T_0$ and its inverse $(\slashed D_T^{\ast}\slashed D_T)^{-1}\colon L^2(N_T,E_T^+)\to L^2(N_T,E_T^+)$ satisfies
\begin{equation*}
\| (\slashed D_T^{\ast}\slashed D_T)^{-1} \|_{\mathcal L(L^2,L^2)} \leq CT^2. 
\end{equation*}
\end{theorem}

\begin{proof}
This is precisely the statement shown in \cite[\textsection 5.B]{ni02} (there with the roles of the subbundles $E_T^+$ and $E_T^-$ interchanged).	
\end{proof}

We remark  that the statement of the theorem and its proof extend straightforwardly  to the case of a manifold glued in the above way from  a finite number of  manifolds with cylindrical ends. Let us now explain how the  family of   Riemann surfaces $\Sigma_t$   fits into this general framework. We first  recall that the noded Riemann surface $\Sigma_0$ is endowed with a Riemannian metric which with respect to the complex coordinate $z$ near each puncture $p\in\mathfrak p$ takes the form $g=\frac{|dz|^2}{|z|^2}$. Passing to the cylindrical coordinates $(\tau,\vartheta)=(-\log r,-\theta)$ shows that $g=d\tau^2+d\vartheta^2$ and hence each connected component of $(\Sigma_0,g)$ is a Riemannian manifold with cylindrical ends of the form considered above. The discussion at the end of \textsection \ref{subsect:degeneratingfamilies} makes it clear that the pair $(E,H)$ is a cylindrical hermitian  vector bundle over $\Sigma_0$. Similarly, the bundle of $\mathfrak{sl}(E)$-valued differential forms is a  cylindrical hermitian vector bundle; it decomposes into  the subbundles of forms of  even, respectively odd total degree, and hence is $\Z_2$-graded.   Also, the gluing construction described  in \textsection \ref{subsect:degeneratingfamilies} is a special case of the one  above   for general cylindrical vector bundles. 
\medskip\\
Let us now introduce the relevant  Dirac-type operator and check that it is indeed an exponentially small perturbation of a cylindrical Dirac-type operator. We let $(A,\Phi)$ denote the exact solution of Eq.\ \eqref{eq:hitequ} on the noded surface $\Sigma_0$  as initially fixed in \textsection \ref{subsect:approxsolutions}. Its  gives rise to  the elliptic complex 
\begin{multline}\label{eq:ellcomplex}
0\longrightarrow \Omega^0(\Sigma_0,\mathfrak{su}(E))\stackrel{L_1}{\longrightarrow} \Omega^1(\Sigma_0,\mathfrak{su}(E))\oplus\Omega^{1,0}(\Sigma_0,\mathfrak{sl}(E))\\
\stackrel{L_2}{\longrightarrow} \Omega^2(\Sigma_0,\mathfrak{su}(E))\oplus \Omega^{2}(\Sigma_0,\mathfrak{sl}(E)) \longrightarrow 0,
\end{multline}
where
\begin{equation*}
L_1\gamma=(d_A\gamma,[\Phi\wedge\gamma])
\end{equation*}
and $L_2=D\mathcal H$ is the differential of the Hitchin operator as in Eq.\ \eqref{eq:diffHitchinop}, i.e.\
\begin{equation*}
L_2(\alpha,\varphi)=\begin{pmatrix}d_A\alpha+[\Phi\wedge\varphi^{\ast}]+[\Phi^{\ast}\wedge\varphi]\\
\bar\partial_A\varphi+[\Phi\wedge\alpha]\end{pmatrix}.
\end{equation*}
Decomposing   $\Omega^{\ast}(\Sigma,\mathfrak{sl}(E))$ into forms of even, respectively odd total degree, we can combine the operators $L_1$ and $L_2$  into the Dirac-type operator $\mathfrak D_{\infty}$ on $\Sigma_0$ given by
\begin{equation}\label{eq:dirtypeopdinfty}
\mathfrak D_{\infty}:=\begin{pmatrix} 0& L_1^{\ast}+L_2\\
L_1+L_2^{\ast}&0\end{pmatrix}.
\end{equation}
Following Hitchin \cite[pp.\ 85--86]{hi87}, and using  his notation, it is convenient  to identify
\begin{equation*}
\Omega^2(\Sigma_0,\mathfrak{su}(E))\cong	 \Omega^0(\Sigma_0,\mathfrak{su}(E))\qquad \textrm{and}\qquad \Omega^2(\Sigma_0,\mathfrak{sl}(E))\cong	 \Omega^0(\Sigma_0,\mathfrak{sl}(E))
\end{equation*}
using the Hodge-$\ast$ operator, and
\begin{equation*}
\Omega^1(\Sigma_0,\mathfrak{su}(E))\cong	 \Omega^{0,1}(\Sigma_0,\mathfrak{sl}(E))
\end{equation*}
via the projection  $A\mapsto  \pi^{0,1}A$ to the $(0,1)$-component. We further identify $(\gamma_1,\gamma_2)\in \Omega^0(\Sigma_0,\mathfrak{su}(E))\oplus \Omega^0(\Sigma_0,\mathfrak{su}(E))$ with $\psi_1=\gamma_1+i\gamma_2\in \Omega^0(\Sigma_0,\mathfrak{sl}(E))$. With these identifications understood, the operator $L_1+L_2^{\ast}$ is the map
\begin{align}\label{eq:operatorl1l2ast}
\nonumber L_1+L_2^{\ast}\colon \Omega^{0}(\Sigma_0,\mathfrak{sl}(E))\oplus  \Omega^{0}(\Sigma_0,\mathfrak{sl}(E))\to  \Omega^{0,1}(\Sigma_0,\mathfrak{sl}(E))
\oplus  \Omega^{1,0}(\Sigma_0,\mathfrak{sl}(E)),\\
(\psi_1,\psi_2)\mapsto \begin{pmatrix}
 \bar\partial_A\psi_1+[\Phi^{\ast}\wedge\psi_2]\\
\partial_A\psi_2+[\Phi\wedge\psi_1]
\end{pmatrix}.
\end{align}
As shown in Proposition \ref{prop:expsmallperturbation} below, the operator $\mathfrak D_{\infty}$ arises as an exponentially small perturbation of the operator
\begin{equation*}
\hat{\mathfrak D}_{\infty}=\begin{pmatrix} 0& \hat L_1^{\ast}+\hat L_2\\
\hat L_1+\hat L_2^{\ast}&0\end{pmatrix}.
\end{equation*}
defined in the same way, with $(A,\Phi)$ replaced by some model solution  \begin{equation*}
(A^{\mod},\Phi^{\mod})=\Big(\beta( \frac{dz}{z}- \frac{d\bar z}{\bar z}),\varphi  \frac{dz}{z}\Big)
\end{equation*}
as in Eq.\ \eqref{eq:modsol} along each cylindrical neck. We check that the latter operator is in fact cylindrical. Introducing the complex coordinate $\zeta=\tau+ i \vartheta$, we have the identities
\begin{equation*}
d\tau = -\frac{dr}{r},\quad d\theta=-d\vartheta,\quad  \frac{dz}{z}=-d\zeta,	\quad  \frac{d\bar z}{\bar z}=-d\bar \zeta.
\end{equation*}
Hence $\hat L_1+\hat L_2^{\ast}$ (and similarly $\hat L_1+\hat L_2^{\ast}$) can indeed be written as the cylindrical differential operator
\begin{multline*}
\hat L_1+\hat L_2^{\ast}=\frac{\sqrt 2}{2} G(\partial_{\tau} - D) \colon\\
 (\psi_1,\psi_2)\mapsto \frac{1}{2}\begin{pmatrix}
 \partial_{\tau} \psi_1 d\bar \zeta\\
  \partial_{\tau} \psi_2  d\zeta
\end{pmatrix}	- \begin{pmatrix}
(\frac{i}{2}\partial_{\vartheta} \psi_1+ [\beta,\psi_1]- [\varphi^{\ast },\psi_2] ) d\bar \zeta \\
(-\frac{i}{2} \partial_{\vartheta} \psi_2- [\beta,\psi_2]- [\varphi,\psi_2] ) d\zeta 
\end{pmatrix}.	
\end{multline*} 
Here
\begin{equation}\label{eq:operatorD}
D(\psi_1,\psi_2)= 2\begin{pmatrix}
\frac{i}{2}\partial_{\vartheta} \psi_1+ [\beta,\psi_1]- [\varphi^{\ast },\psi_2]   \\
-\frac{i}{2} \partial_{\vartheta} \psi_2- [\beta,\psi_2]- [\varphi,\psi_2]  \end{pmatrix}
\end{equation}
and $G(\psi_1,\psi_2)=\frac{\sqrt 2}{2} (\psi_1 d\bar \zeta,\psi_2  d\zeta)$ denotes Clifford multiplication by $d\tau$.

\begin{proposition}\label{prop:expsmallperturbation}
The difference $B:=L_1+L_2^{\ast}- \hat L_1- \hat L_2^{\ast}$  satisfies the estimate stated in Eq.\ \eqref{eq:expdecaycond}, i.e.\ $B$ is exponentially decaying in $\tau$.
\end{proposition}

\begin{proof}
We set
\begin{equation*}\label{eq:expdecaycond1}
(A^{\app},\Phi^{\app})=(A^{\mod},\Phi^{\mod})+\Big(\beta_1( \frac{dz}{z}- \frac{d\bar z}{\bar z}),\varphi_1  \frac{dz}{z}\Big)
\end{equation*}
for suitable endomorphism valued maps $\beta_1$ and $\varphi_1$.  Lemma \ref{thm:asymptbehav} shows that $\beta_1,\varphi_1\in C_{\delta}^0$ for some $\delta>0$. It follows that the operator $B$, which is given by
\begin{equation*}
B(\psi_1,\psi_2)= \begin{pmatrix}
  (-[\beta_1,\psi_1]+ [\varphi_1^{\ast },\psi_2] )d\bar \zeta  \\
(  [\beta_1,\psi_2]+ [\varphi_1,\psi_2] )d\zeta \end{pmatrix},
\end{equation*}
satisfies the asserted exponential decay estimate with respect to cylindrical coordinates.
\end{proof}

Our next goal is to show that the space $\ker(L_1+L_2^{\ast})\cap L_{\operatorname{ext}}^2(\Sigma_0)$ is trivial. We prepare this result with the following proposition.

\begin{proposition}\label{prop:asympttrace}
Suppose  $(\psi_1,\psi_2)\in \ker(L_1+L_2^{\ast})\cap L_{\operatorname{ext}}^2(\Sigma_0)$. Then its asymptotic trace $\partial_{\infty}(\psi_1,\psi_2)$ is of the form
\begin{equation*}
\partial_{\infty}(\psi_1,\psi_2)=\left( \begin{pmatrix}
 c_1&0 \\0&-c_1
\end{pmatrix},\begin{pmatrix}
 c_2&0 \\0&-c_2
\end{pmatrix}  \right)	
\end{equation*}
for constants $c_1,c_2\in\C$.
\end{proposition}

\begin{proof}
By \cite[p.\ 169]{ni02},  the space $\partial_{\infty}\ker(L_1+L_2^{\ast})$ of asymptotic traces is a subspace of $\ker D$ (with $D$ the operator in Eq.\ \eqref{eq:operatorD}), and hence it suffices  to check that the elements of the latter  have the asserted form. Passing to the Fourier decomposition $(\psi_1,\psi_2)=(\sum_{j\in\Z}\psi_{1,j}e^{ij\vartheta},\sum_{j\in\Z}\psi_{2,j}e^{ij\vartheta})$, where $\psi_{i,j}\in\mathfrak{sl}(2,\C)$, the equation $D(\psi_1,\psi_2)=0$ is equivalent to the   system of linear equations
\begin{equation}
\begin{pmatrix}\label{eq:linearsystem}
-\frac{j}{2}  \psi_{1,j} + [\beta,\psi_{1,j}]- [\varphi^{\ast },\psi_{2,j}]  \\
\frac{j}{2}\psi_{1,j}  - [\beta,\psi_{2,j}]- [\varphi,\psi_{1,j}]  
\end{pmatrix}=0
\end{equation}
for  $j\in\Z$. By assumption, the endomorphisms $\beta=\begin{pmatrix} \alpha&0\\0&-\alpha\end{pmatrix}$ and $\varphi=\begin{pmatrix} C&0\\0&-C\end{pmatrix}$ are diagonal.  It follows that $D$ acts invariantly on diagonal, respectively off-diagonal endomorphisms, and hence it suffices to consider both cases separately. Suppose first that
\begin{equation*}
(\psi_{1,j},\psi_{2,j})=\left(\begin{pmatrix}
c_{1,j}&0  \\0& -c_{1,j}  
\end{pmatrix},\begin{pmatrix}
c_{2,j}&0  \\0& -c_{2,j}  
\end{pmatrix}\right)
\end{equation*}
is diagonal. Then Eq.\ \ref{eq:linearsystem} has a non-trivial solution if and only if $j=0$, which  is then of the asserted form. Now let
\begin{equation*}
(\psi_{1,j},\psi_{2,j})=\left(\begin{pmatrix}
0&d_{1,j} \\e_{1,j} & 0
\end{pmatrix},\begin{pmatrix}
0&d_{2,j} \\e_{2,j}  & 0  
\end{pmatrix}\right)
\end{equation*}
for  $d_{i,j},e_{i,j}\in\C $. Then Eq.\ \ref{eq:linearsystem} is equivalent to
\begin{equation*}
\begin{pmatrix} 
-\frac{j}{2} +2\alpha    & -2\bar C \\
-2C &\frac{j}{2}  - 2\alpha
\end{pmatrix}\begin{pmatrix} d_{1,j}\\ d_{2,j} \end{pmatrix}  =0,
\end{equation*}
with a  similar linear equation being satisfied by    $(e_{1,j},e_{2,j})$. The determinant of the above matrix is $-(\frac{j}{2}  - 2\alpha)^2-4 C\bar C <0$, using  that  $C\neq0$ as  satisfied by assumption (A1). It hence does not admit any non-trivial solution, proving the proposition.
\end{proof}

\begin{lemma}\label{lemma:limitkernelnew}
The operator $L_1+L_2^{\ast}$, considered as a densely defined operator on $L_{\operatorname{ext}}^2(\Sigma_0)$, has trivial kernel.
\end{lemma}

\begin{proof}
We divide the proof into two steps.
\begin{step}
Suppose $(\psi_1,\psi_2)\in \ker(L_1+L_2^{\ast})\cap L_{\operatorname{ext}}^2(\Sigma_0)$. Then $d_A\psi_j= [\Phi  \wedge   \psi_j ]= [\Phi^{\ast} \wedge   \psi_j]=0$ for $j=1,2$.
\end{step}

For a closed Riemann surface, this statement is shown in \cite[pp.\ 85--86]{hi87}. It carries over with some modifications to the present case of a noded Riemann surface $\Sigma_0$.  By Eq.\ \eqref{eq:operatorl1l2ast},  $\psi_1+ \psi_2\in \ker(L_1+L_2^{\ast})$ if and only if
\begin{equation}\label{eq:systempsi}
\begin{cases}
0=& \bar\partial_A\psi_1+[\Phi^{\ast}\wedge\psi_2],\\
0=& \partial_A\psi_2+[\Phi\wedge\psi_1].	
\end{cases}
\end{equation} 
Differentiating the first equation and using that $ \partial_A\Phi^{\ast}=0$  yields 
\begin{equation*}
0= \partial_A\bar\partial_A\psi_1-[\Phi^{\ast}\wedge  \partial_A\psi_2] =\partial_A\bar\partial_A\psi_1+[\Phi^{\ast}\wedge  [\Phi \wedge   \psi_1]].
\end{equation*}
It follows that
\begin{eqnarray*}
\partial\langle \bar\partial_A\psi_1,\psi_1\rangle	&=&\langle \partial_A\bar\partial_A\psi_1,\psi_1 \rangle -\langle \bar\partial_A\psi_1,\bar\partial_A\psi_1\rangle \\
&=&- \langle [\Phi^{\ast}\wedge  [\Phi \wedge   \psi_1]],\psi_1\rangle - \langle \bar\partial_A\psi_1,\bar\partial_A\psi_1\rangle \\
&=& - \left|  [\Phi \wedge   \psi_1] \right|^2 -\left|\bar\partial_A\psi_1 \right|^2.
\end{eqnarray*}
Similarly, using the general identity
\begin{equation*}
\bar\partial_A\partial_A\psi  + \partial_A\bar\partial_A\psi=[F_A  \wedge   \psi]
\end{equation*}
for   $\psi\in  \Omega^{0}(\Sigma_0,\mathfrak{sl}(E))$ together with the equation $F_A+[\Phi  \wedge   \Phi^{\ast}]=0$, we obtain that
\begin{eqnarray*}
\lefteqn{\bar \partial\langle \partial_A\psi_1,\psi_1\rangle}\\
	&=&\langle \bar\partial_A\partial_A\psi_1,\psi_1 \rangle -\langle  \partial_A\psi_1, \partial_A\psi_1\rangle \\
&=& \langle [F_A   \wedge   \psi_1],\psi_1\rangle -\langle \partial_A\bar\partial_A\psi_1,\psi_1 \rangle -\langle  \partial_A\psi_1, \partial_A\psi_1\rangle \\
&=& - \langle  [[\Phi \wedge   \Phi^{\ast}]\wedge \psi_1],\psi_1\rangle +\langle [\Phi^{\ast}\wedge  [\Phi \wedge   \psi_1]],\psi_1\rangle -\langle  \partial_A\psi_1, \partial_A\psi_1\rangle\\
&=& -\left|  [\Phi^{\ast} \wedge   \psi_1] \right|^2  -\left| \partial_A\psi_1 \right|^2.
\end{eqnarray*}
Here the Jacobi identity has been used in the last step. Adding both equations yields
\begin{equation}\label{eq:delplusdelbar}
\begin{split}
\partial\langle \bar\partial_A\psi_1,\psi_1\rangle + \bar \partial\langle \partial_A\psi_1,\psi_1\rangle=\\
  -\left|\bar\partial_A\psi_1 \right|^2 -\left| \partial_A\psi_1 \right|^2 - \left|  [\Phi \wedge   \psi_1] \right|^2-\left|  [\Phi^{\ast} \wedge   \psi_1] \right|^2.
\end{split}  
\end{equation}
We aim to show  that the integral over $\Sigma_0$ of the left-hand side vanishes, which requires  to introduce   some notation.  For   $p\in\mathfrak p$, denote by $\mathcal C_p^{\pm}(0)$ the two connected components of $\mathcal C_p(0)$. We endow these  half-infinite cylinders   with cylindrical  coordinates $(\tau,\vartheta)$, $0\leq\tau<\infty$, as before. For $S>0$ we denote by $\mathcal C_p^{\pm}(S)$ the subcylinders of points $(\tau,\vartheta)\in \mathcal C_p^{\pm}(0)$ with $ \tau\geq S$. Let $\Sigma_S:=\Sigma_0\setminus\bigcup_{p\in\mathfrak p}\mathcal C_p^{\pm}(S)$. By Stoke's theorem, it  follows that
\begin{multline*}\label{eq:greenid}
\int_{\Sigma_S} \partial\langle \bar\partial_A\psi_1,\psi_1\rangle + \bar \partial\langle \partial_A\psi_1,\psi_1\rangle =  \int_{\partial \Sigma_S}\langle \bar\partial_A\psi_1,\psi_1\rangle +  \langle \partial_A\psi_1,\psi_1\rangle \\
=  \int_{\partial \Sigma_S}\langle d_A\psi_1,\psi_1\rangle.
\end{multline*}
Now as $S\to\infty$, $\left.\psi_1\right|_{\tau=S}$ $L^2$-converges to its asymptotic trace $\partial_{\infty}\psi_1 \in \Omega^0(S^1,\mathfrak{sl}(2,\C))$, which by Proposition \ref{prop:asympttrace} is of the form 
\begin{equation*}
\psi_1(\infty)=\begin{pmatrix}
c_1&0\\0&-c_1	
\end{pmatrix}
\end{equation*}
for some constant $c_1\in\C$. Hence $d_A(\partial_{\infty}\psi_1(\infty))=0$ and we conclude that 
\begin{equation*}
\int_{\Sigma_0} \partial\langle \bar\partial_A\psi_1,\psi_1\rangle + \bar \partial\langle \partial_A\psi_1,\psi_1\rangle=\lim_{S\to\infty}\int_{\partial \Sigma_S}\langle d_A\psi_1,\psi_1\rangle=0.
\end{equation*}
Eq.\ \eqref{eq:delplusdelbar} now shows that
\begin{equation*}
 \bar\partial_A\psi_1    =  \partial_A\psi_1=   [\Phi  \wedge   \psi_1] =   [\Phi^{\ast} \wedge   \psi_1] =0.
\end{equation*}
Next,  taking the hermitian adjoint of Eq.\ \eqref{eq:systempsi}, we obtain 
\begin{equation*} 
\begin{cases}
0=& \partial_A\psi_1^{\ast}-[\Phi \wedge\psi_2^{\ast}],\\
0=& \bar\partial_A\psi_2^{\ast}-[\Phi^{\ast}\wedge\psi_1^{\ast}].	
\end{cases}
\end{equation*}
Thus after replacing the solution $(A,\Phi)$ of the self-duality equations with $(A,-\Phi)$, we obtain by the same reasoning as before that 
\begin{equation*}
 \bar\partial_A\psi_2^{\ast}    =  \partial_A\psi_2^{\ast} =   [\Phi  \wedge   \psi_2^{\ast}] =   [\Phi^{\ast} \wedge   \psi_2^{\ast}] =0.
\end{equation*}
Altogether we have thus shown that $d_A\psi_j= [\Phi  \wedge   \psi_j ]= [\Phi^{\ast} \wedge   \psi_j]=0$ for $j=1,2$, as claimed.

\begin{step}
We prove the lemma.
\end{step}

Let $(\psi_1,\psi_2)\in \ker(L_1+L_2^{\ast})\cap L_{\operatorname{ext}}^2(\Sigma_0)$, hence $d_A\psi_j= [\Phi  \wedge   \psi_j ]= [\Phi^{\ast} \wedge   \psi_j]=0$ for $j=1,2$ by Step 1. We prove that $\psi_1=0$ by showing separately that  $\lambda_1:=\psi_1+\psi_1^{\ast}\in \Omega^0(\Sigma_0,i\mathfrak{su}(E))$ and $\lambda_2:=i(\psi_1-\psi_1^{\ast})\in \Omega^0(\Sigma_0,i\mathfrak{su}(E))$ vanish. The reasoning for $\psi_2$ is entirely similarly. Differentiation shows that
\begin{equation*}
d \left|\lambda_1\right|^2=2\langle d_A\lambda_1,\lambda_1\rangle=0,
\end{equation*}
and hence  the function  $|\lambda_1|^2\equiv C$ is constant on $\Sigma_0$. We further notice that $[\Phi  \wedge   \psi_1 ]= [\Phi^{\ast} \wedge   \psi_1]=0$ implies   $[\Phi  \wedge   \lambda_1 ] =0$.  Hence  $\lambda_1(x)\in\ker M_{\Phi(x)}$ for every $x\in\Sigma_0$, where $M_{\Phi}\colon \Omega^0(\Sigma_0,i\mathfrak{su}(E))\to \Omega^0(\Sigma_0,i\mathfrak{su}(E))$  is the linear operator defined in Eq.\ \eqref{eq:defnlinoperatorMphi}.  We claim that $\ker M_{\Phi(x_0)}=\{0\}$ for at least one $x_0\in\Sigma_0$.  By  the assumption (A3) made in \textsection \ref{subsect:approxsolutions},  $\det\Phi$ has a simple zero in at least one point $x_0\in\Sigma_0$. With respect to a local holomorphic coordinate and a local trivialization of $E$,   we may write $\Phi(x_0)=\varphi(x_0)\, dz^2$ for some endomorphism $\varphi(x_0)\in \SL(2,\C)$ satisfying $\det \varphi(x_0)=0$. Now $ \varphi(x_0)\neq 0$ for otherwise the Taylor expansion of $ \varphi$ around $x_0$ would start with linear or higher order terms, and then that of $\det \varphi$  would not involve a linear term, contradicting the simplicity of the zero.  Moreover, we claim that $[\varphi(x_0),\varphi(x_0)^{\ast}]\neq0$, i.e.\ that $\varphi(x_0)$ is not normal. Since conjugation by a unitary endomorphism preserves this property, it suffices to assume that $\varphi(x_0)$ has upper triangular form, and hence is equal to
\begin{equation*}
\varphi(x_0)=\begin{pmatrix}	
0&\nu\\0&0	
\end{pmatrix}
\end{equation*}
for some $\nu\in\C\setminus\{0\}$. It follows that
\begin{equation*}
[\varphi(x_0),\varphi(x_0)^{\ast}]=\begin{pmatrix}	
|\nu|^2&0\\0&-|\nu|^2	
\end{pmatrix}\neq0.
\end{equation*}
Lemma \ref{lem:kernelMphi} now shows that   $\ker M_{\varphi(x_0)}$ is trivial.  It follows that    $C=\left|\lambda_1(x_0)\right|^2=0$ and therefore  $\lambda_1$ has to vanish identically. That   $\lambda_2=0$ follows along the same line of arguments. This finishes the proof of   the lemma.
\end{proof}

Recall the construction of the approximate solution $(A_R^{\app},\Phi_R^{\app})$ to parameter $0<R<1$ as described in \textsection \ref{subsect:approxsolutions}. As before, let $T=-\log R$. Replacing $(A,\Phi)$ with  $(A_R^{\app},\Phi_R^{\app})$ in the elliptic complex in Eq.\ \eqref{eq:ellcomplex}, we obtain  the $\Z_2$-graded Dirac-type operator
\begin{equation*}
\mathfrak D_T:=\begin{pmatrix} 0& L_{1,T}^{\ast}+L_{2,T}\\
L_{1,T}+L_{2,T}^{\ast}&0\end{pmatrix}
\end{equation*}
on the closed surface $\Sigma_t$. Note that for  $T=\infty$ this equals the operator $\mathfrak D_{\infty}$ in Eq.\ \eqref{eq:dirtypeopdinfty}. We set
\begin{equation}\label{eq:Laplacianapp}
\begin{split}
\mathcal D_T:=(L_{1,T}+L_{2,T}^{\ast})^{\ast}(L_{1,T}+L_{2,T}^{\ast})\\
=L_{1,T}^{\ast}L_{1,T} +L_{2,T}L_{1,T}+L_{1,T}^{\ast}L_{2,T}^{\ast}+L_{2,T} L_{2,T}^{\ast}.
\end{split}
\end{equation}
In general, it does not hold that $L_{2,T}L_{1,T}=0$ since $(A_R^{\app},\Phi_R^{\app})$ need not be an exact solution. The following lemma is now an immediate consequence of Theorem 	\ref{thm:applicationclm}.

\begin{lemma}\label{lem:mainestimate}
There  exist constants $T_0>0$ and   $C>0$  such that the operator $\mathcal D_T$ is bijective for all $T>T_0$ and its inverse $\mathcal D_T^{-1}\colon L^2(\Sigma_t)\to L^2(\Sigma_t)$ satisfies
\begin{equation*}
\| \mathcal D_T^{-1} \|_{\mathcal L(L^2,L^2)} \leq CT^2. 
\end{equation*}
\end{lemma}

\begin{proof}
The statement follows by applying Theorem 	\ref{thm:applicationclm}. The assumption that the kernel of the   limiting operator $L_{1}+L_{2}^{\ast}$ is trivial on $L_{\operatorname{ext}}^2(\Sigma_0)$ is satisfied thanks to Lemma \ref{lemma:limitkernelnew}. 
\end{proof}

After these preparations we are finally in the position to show the desired estimate on the inverse $G_{(A_R^{\app},\Phi_R^{\app})}$ of the operator
\begin{equation*}\label{eq:defnlinoperator1}
L_{(A_R^{\app},\Phi_R^{\app})}=\Delta_{A }-i \ast M_{\Phi }\colon \Omega^0(\Sigma_t,i\mathfrak{su}(E))\to \Omega^0(\Sigma_t,i\mathfrak{su}(E))
\end{equation*}
introduced in Eq.\ \eqref{eq:defnlinoperator}.

\begin{theorem}\label{thm:L2est}
There exists a constant $R_0>0$   such that for every $0<R<R_0$  the operator norm of $G_{(A_R^{\app},\Phi_R^{\app})}\colon L^2(\Sigma_t)\to  L^2(\Sigma_t)$ is bounded from above by some constant $M_R  \sim \left|\log R\right|^{2}$.
\end{theorem}

\begin{proof}
We first observe that conjugation by the map $\gamma\mapsto i\gamma$ transforms  $L_{(A_R^{\app},\Phi_R^{\app})}$ into a  unitarily equivalent operator (also denoted by $L_{(A_R^{\app},\Phi_R^{\app})}$) acting on the space $\Omega^0(\Sigma_t,\mathfrak{su}(E))$. It suffices to show the statement for this operator. Let $T=-\log R$ as before. We   note that the difference $ L_{(A_R^{\app},\Phi_R^{\app})} - L_{1,T}^{\ast}L_{1,T}$ is a nonnegative operator. Indeed, using Lemma \ref{positivity} it follows for all  $\gamma\in \Omega^0(\Sigma_t,\mathfrak{su}(E))$ that
\begin{equation*}
\langle ( L_{(A_R^{\app},\Phi_R^{\app})} - L_{1,T}^{\ast}L_{1,T})\gamma,\gamma\rangle =  \| [\Phi_R^{\app}\wedge \gamma] \|_{ L^2(\Sigma_t)}^2\geq0.
\end{equation*}
It is therefore sufficient  to prove the asserted upper bound for the norm of the inverse of $L_{1,T}^{\ast}L_{1,T}$.  Equivalently, we show that there exist constants $T_0>0$ and $C>0$ such that for all $T>T_0$ and   $\gamma\in H^2(\Sigma_t)$ it holds that
\begin{equation}\label{eq:defnlinoperator2}
\| L_{1,T}^{\ast}L_{1,T}\gamma	\|_{L^2(\Sigma_t)}\geq CT^{-2}\| \gamma	\|_{L^2(\Sigma_t)}.
\end{equation}
Lemma \ref{lem:mainestimate} yields the existence of constants $T_0>0$ and $C>0$  such that
\begin{equation*} 
\| \mathcal D_T \psi 	\|_{L^2(\Sigma_t)}\geq CT^{-2}\|\psi 	\|_{L^2(\Sigma_t)}
\end{equation*}
for  every $T>T_0$ and every differential form of even degree  $ \psi\in H^2(\Sigma_t)$, where $\mathcal D_T$ is the operator in Eq.\ \eqref{eq:Laplacianapp}.  We now note that
\begin{equation*} 
L_{2,T}L_{1,T}\gamma = [F_{A_R^{\app}}^{\perp} \wedge\gamma]+[[\Phi_R^{\app}\wedge (\Phi_R^{\app})^{\ast}]\wedge \gamma] 
\end{equation*}
for all $\gamma\in \Omega^0(\Sigma_t,\mathfrak{su}(E))$. By Lemma \ref{lem:errorappr} this term satisfies the estimate
\begin{equation*}
\| L_{2,T}L_{1,T}\gamma 	\|_{L^2(\Sigma_t)}\leq C_1R^{\delta''}\|  \gamma 	\|_{L^2(\Sigma_t)}= C_1e^{-\delta''T}\|  \gamma 	\|_{L^2(\Sigma_t)}
\end{equation*}
for $T$-independent constants $C_1,\delta''>0$. As follows from Eq.\ \eqref{eq:Laplacianapp}, the $0$-form $\gamma$ satisfies  
\begin{eqnarray*}
L_{1,T}^{\ast}L_{1,T}\gamma = \mathcal D_T\gamma -L_{2,T}L_{1,T}\gamma.
\end{eqnarray*}
Hence the triangle inequality yields the estimate
\begin{eqnarray*} 
\| L_{1,T}^{\ast}L_{1,T}\gamma 	\|_{L^2(\Sigma_t)}&\geq&  \| \mathcal D_T\gamma 	\|_{L^2(\Sigma_t)}-\| L_{2,T}L_{1,T}\gamma\|_{L^2(\Sigma_t)}\\
&\geq& CT^{-2}\|\gamma	\|_{L^2(\Sigma_t)}- C_1e^{-\delta''T}\|  \gamma 	\|_{L^2(\Sigma_t)},
\end{eqnarray*}
which for $T$ sufficiently large implies the desired inequality \eqref{eq:defnlinoperator2}. This finishes the proof of the theorem.
\end{proof}

From now on, we set $L_R:=L_{(A_R^{\app},\Phi_R^{\app})}$. Our next goal is to obtain an upper bound for the norm of the operator 
\begin{equation*}
G_R=L_R^{-1}\colon L^2(\Sigma_t)\to H^2(\Sigma_t).
\end{equation*}
 To this aim, we fix as a reference connection  some smooth extension $B$      of the model connection $A_R^{\mod}$ to   $\Sigma_t$, i.e.~$B$ equals\begin{equation*}
B=\begin{pmatrix}2i\alpha_p &0\\0&-2i\alpha_p\end{pmatrix}\,d\theta.
\end{equation*}
on each cylinder $\mathcal C_p(R)$. We introduce the Banach space  
\begin{equation*}
H_B^2(\Sigma_t):=\{\gamma\in L^2(\Sigma_t)\mid \nabla_B\gamma,\nabla_B^2\gamma\in L^2(\Sigma_t)\},
\end{equation*}
and similarly the  space $H_B^1(\Sigma_t)$.  We  compare this norm with the equivalent graph norm for the operator $L_R\colon H_B^2(\Sigma_t)\to L^2(\Sigma_t)$ defined by
\begin{equation*}
\|\gamma\|_{L_R}^2:=\|L_R\gamma\|_{L^2(\Sigma_t)}^2+\|\gamma\|_{L^2(\Sigma_t)}^2.
\end{equation*}

\begin{proposition}\label{prop:equivnorms}
There exists a  positive constant $C$ which does not depend on $R$ such that $\|\gamma\|_{H_B^2(\Sigma_t)}\leq C   \|\gamma\|_{L_R}$ for all $\gamma\in H_B^2(\Sigma_t)$.
\end{proposition}

\begin{proof}
We first note   the elliptic estimate
\begin{equation}\label{eq:equivnormsellest}
\|\gamma\|_{H_B^2(\Sigma_t)}\leq C_0(\|\Delta_B\gamma\|_{L^2(\Sigma_t)}+\|\gamma\|_{L^2(\Sigma_t)})
\end{equation}
for $\gamma\in H_B^2(\Sigma_t)$. The constant $C_0$ appearing here can be  taken to be independent of $R$. Namely, after passing to cylindrical coordinates $(\tau,\vartheta)=(-\log r,-\theta)$  on $\mathcal C_p(R)$, $p\in\mathfrak p$,   as before, the operator  $\Delta_B$  becomes invariant under translations $\tau\mapsto \tau_0+\tau$. Hence the above elliptic estimate holds with   a uniform constant  $C$. We now  consider the difference $S_R:=\Delta_{A_R^{\app}}-\Delta_B$.  On   $\mathcal C_p(R)$,   putting $\lambda_R\,d\theta:=A_R^{\app}-B$ and $B=\beta \,d\theta$, it is given by     
\begin{equation*}
S_R\gamma=-r[\partial_{\theta}\lambda_R+[\beta,\lambda_R],\gamma]-2r[\lambda_R,\partial_{\theta}\gamma+[\beta,\gamma]]-r[\lambda_R,[\lambda_R,\gamma]].
\end{equation*}
Then for any  $0<R_0\leq 1$ we let $\mathcal F_p(R_0)$ denote (the possibly empty) subcylinder consisting of the points   $(r,\theta)\in \mathcal C_p(R)$ such that $\{ R\leq r\leq R_0\}$. On $\mathcal F_p(R_0)$, the term $S_R\gamma$     can be estimated as 
\begin{eqnarray*}
\lefteqn{\|S_R\gamma\|_{L^2(\mathcal F_p(R_0))}}\\
&\leq &C (\|\lambda_R\|_{H_B^1(\mathcal F_p(R_0))} \|\gamma\|_{C^0(\mathcal F_p(R_0))} +  \|\lambda_R\|_{H_B^1(\mathcal F_p(R_0))}\|\gamma\|_{H_B^2(\mathcal F_p(R_0))} \\
&& +  \|\lambda_R\|_{H_B^1(\mathcal F_p(R_0))}^2\|\gamma\|_{H_B^1(\mathcal F_p(R_0))}   )\\
&\leq& C_1R_0^{\nu}\|\gamma\|_{H_B^2(\mathcal F_p(R_0))}
\end{eqnarray*}
for   constants  $C_1,\nu>0$, using Lemma \ref{thm:asymptbehav} and Lemma \ref{lem:normformlem}. We now fix $R_0$ sufficiently small such that $C_0C_1R_0^{\nu}<1$ and define   $\Sigma_t^
{\operatorname{ext}}:=\Sigma_t\setminus\bigcup_{p\in\mathfrak p}\mathcal  F_p(R_0)$. On $\Sigma_t^
{\operatorname{ext}}$ we can estimate 
\begin{equation*}
 \|S_R\gamma\|_{L^2(\Sigma_t^
{\operatorname{ext}})}	\leq C( \|\nabla_B\gamma\|_{L^2(\Sigma_t^
{\operatorname{ext}})}  +\|\gamma\|_{L^2(\Sigma_t^
{\operatorname{ext}})} )
\end{equation*}
for some $R$-independent constant $C$, since there the coefficients of the first-order operator $S_R$ admit uniform bounds in $R$ with respect to the $C^1$ norm. Now   integration by parts over the closed surface $\Sigma_t$ and an application of the Cauchy-Schwarz inequality yields  
\begin{multline*}
\|\nabla_B\gamma\|_{L^2(\Sigma_t^
{\operatorname{ext}})}^2 \leq \|\nabla_B\gamma\|_{L^2(\Sigma_t)}^2 = \langle\Delta_B\gamma,\gamma \rangle_{L^2(\Sigma_t)} \\
\leq  \frac{\varepsilon^2}{2}\| \Delta_B\gamma \|_{L^2(\Sigma_t)}^2  +\frac{1}{2\varepsilon^2}  	\|\gamma	\|_{L^2(\Sigma_t)}^2
\end{multline*}
for any $\varepsilon>0$. Combining the estimates on the sets $\mathcal F_p(R_0)$ and $\Sigma_t^
{\operatorname{ext}}$ we thus arrive at
\begin{equation}\label{eq:unifestSR}
\|S_R\gamma\|_{L^2(\Sigma_t)} \leq CR_0^{\nu}\|\gamma\|_{H_B^2(\mathcal C_p(R))}+  \frac{\varepsilon}{\sqrt 2}\| \Delta_B\gamma \|_{L^2(\Sigma_t)}  +\frac{1}{\sqrt 2\varepsilon}  	\|\gamma	\|_{L^2(\Sigma_t)}.
\end{equation}
Therefore,  by inequality \eqref{eq:equivnormsellest} and the triangle inequality,
\begin{eqnarray}\label{eq:unifequivnorms}
 \nonumber\|\gamma\|_{H_B^2(\Sigma_t)}&\leq& C_0 (\|\Delta_B\gamma\|_{L^2(\Sigma_t)}+\|\gamma\|_{L^2(\Sigma_t)})\\
\nonumber&\leq& C_0 (\|\Delta_{A_R^{\app}}\gamma\|_{L^2(\Sigma_t)}+\|S_R\gamma\|_{L^2(\Sigma_t)}+\|\gamma\|_{L^2(\Sigma_t)})\\
\nonumber&\leq& C_0 (\|L_R\gamma\|_{L^2(\Sigma_t)}+\|S_R\gamma\|_{L^2(\Sigma_t)}+\|\gamma\|_{L^2(\Sigma_t)})\\
&\leq& C_0( C_2 \|\gamma\|_{L_R}+\|\ast M_{\Phi_R^{\app}}\gamma\|_{L^2(\Sigma_t)}+C_1R_0^{\nu}\|\gamma\|_{H_B^2(\Sigma_t)}).
\end{eqnarray}
In the last step we   used inequality \eqref{eq:unifestSR} to bound the term $\|S_R\gamma\|_{L^2(\Sigma_t)}$. Then, by choosing the constant $\varepsilon>0$ sufficiently small, the term  involving $\Delta_B\gamma$ can be absorbed in $\|\gamma\|_{H_B^2(\Sigma_t)}$. The right-hand side of inequality \eqref{eq:unifequivnorms} can   further be estimated as follows. The  summand $C_0C_1R_0^{\nu}\|\gamma\|_{H_B^2(\Sigma_t)}$   can be absorbed in $\|\gamma\|_{H_B^2(\Sigma_t)}$  since $C_0C_1R_0^{\nu}<1$ by assumption. Finally, the norm of the operator $\ast M_{\Phi_R^{\app}}\colon L^2(\Sigma_t)\to L^2(\Sigma_t)$ is  bounded from above by   some uniform constant $C$. To see this, we recall that $\Phi_R^{\app} =\varphi_p\frac{dz}{z}$ on  $\mathcal C_p(R)$ for some pointwise bounded endomorphism field $\varphi_p$. Hence   $M_{\Phi_R^{\app}}$ acts as
\begin{equation*}
M_{\Phi_R^{\app}}\gamma= ([\varphi_p^{\ast},[\varphi_p,\gamma]] + [ \varphi_p , [ \varphi_p^{\ast},\gamma]] )\frac{d\bar z\wedge dz}{|z|^2}.
\end{equation*} 
With respect to the metric $g=\frac{|dz|^2}{|z|^2}$, we have that $\ast \frac{d\bar z\wedge dz}{|z|^2}=2i $, hence the endomorphism $\ast M_{\Phi_R^{\app}}$ satisfies a uniform pointwise bound on $\mathcal C_p(R)$, and thus on $\Sigma_t$. This implies the asserted bound for the operator norm and finishes the proof of the proposition. 
\end{proof}

Recall that  $G_R=L_R^{-1}$.

\begin{corollary}\label{lem:estL2H2norm}
There exists a constant $C>0$   such that for all sufficiently small $0<R<R_0$ there holds the estimate
\begin{equation*}
\|G_R\gamma\|_{H_B^2(\Sigma_t)}\leq C(\log R)^2\|\gamma\|_{L^2(\Sigma_t)}
\end{equation*}
for all $\gamma\in L^2(\Sigma_t)$.
\end{corollary}

\begin{proof}
Let $\gamma\in L^2$. From Proposition \ref{prop:equivnorms} and Theorem \ref{thm:L2est} it follows that
\begin{eqnarray*}
\|G_R\gamma\|_{H_B^2(\Sigma_t)}^2&\leq& C\|G_R\gamma\|_{L_R}^2\\
&=&C (\|\gamma\|_{L^2(\Sigma_t)}^2+\|G_R\gamma\|_{L^2(\Sigma_t)}^2)\\
&\leq& C (\|\gamma\|_{L^2(\Sigma_t)}^2+(\log R)^2\|\gamma\|_{L^2(\Sigma_t)}^2)\\
&\leq& C(\log R)^2\|\gamma\|_{L^2(\Sigma_t)}^2,
\end{eqnarray*}
as claimed.
\end{proof}

%
%
\section{Gluing theorem}\label{sect:gluingthm}

%
\subsection{Deforming the approximate solutions}

We are now finally prepared to show that  every approximate solution $(A_R^{\app},\Phi_R^{\app})$  can be perturbed to a nearby exact solution of Hitchin's equations when $0<R<R_0$ is sufficiently small. We keep using the notation introduced in \textsection \ref{lin.hit.ope}, with an additional subscript $R$ to indicate the underlying manifold.
\medskip\\
We aim to find such an exact solution in the same complex gauge orbit as $(A_R^{\app},\Phi_R^{\app})$, i.e.\  to be of the form $(A_R,\Phi_R)=\exp(\gamma_R)^{\ast}(A_R^{\app},\Phi_R^{\app})$ for a suitable map $\gamma_R\in H_B^2(\Sigma_t,i\mathfrak{su}(E))$. To prove existence of      $\gamma_R$ we shall employ a  contraction mapping argument in the manner of \cite{msww14}. To carry it out, we need to study the linearization
$L_R$,  computed at $(A_R^{\app},\Phi_R^{\app})$.  The argument relies on controlling the following quantities: 
\begin{itemize}
\item The error up to which $(A_R^{\app},\Phi_R^{\app})$ satisfies the self-duality equations, i.e.~the norm of the quantity $ F_{A_R^{\app}}^{\perp}+ [\Phi_R^{\app}\wedge (\Phi_R^{\app})^{\ast}]$, or equivalently in the notation of Eq.\ \eqref{eq:defmapFt}, the norm of $\mathcal F_R(0)$. An estimate for this error was obtained in Lemma \ref{lem:errorappr}.
\item The norm of the inverse $G_R=L_R^{-1}$, which is taken care of by Theorem  \ref{thm:L2est} and  Corollary \ref{lem:estL2H2norm}.
\item The Lipschitz constants of the linear and higher order terms in the Taylor expansion of $\mc{F}_R$.
\end{itemize}
We are now focussing on the last issue more closely.
\medskip\\
For any pair $(A,\Phi)$ and $g=\exp(\gamma)$, $\gamma\in\Omega^0(\Sigma_t,i\mf{su}(E))$, we have that
\begin{equation*}
\mathcal{O}_{(A,\Phi)}(\gamma)=g^{\ast}(A,\Phi)=(A+g^{-1}(\bar\partial_{A}g)-(\partial_{A}g)g^{-1},g^{-1}\Phi g),
\end{equation*}
and consequently, 
\begin{gather*}
\exp(\gamma)^{\ast}A=A + (\bar\partial_{A} - \partial_{A}) \gamma + R_{A}(\gamma),\\
\exp(-\gamma)\Phi \exp(\gamma)=\Phi + [\Phi\wedge \gamma] + R_{\Phi}(\gamma).
\end{gather*}
The explicit expressions of these remainder terms are
\begin{align}
\label{eq:remainderRA}
\begin{split}
R_{A}(\gamma)=\exp(-\gamma) (\bar\partial_{A}\exp (\gamma))-(\partial_{A}\exp (\gamma)) 
\exp(-\gamma)- (\bar\partial_{A}-\partial_{A}) \gamma,\\
R_{\Phi}(\gamma)= \exp(-\gamma) \Phi \exp (\gamma) - [\Phi\wedge \gamma]-\Phi.
\end{split}
\end{align}
We then calculate that 
\begin{equation}\label{eq:TaylorexpF}
\begin{array}{rl}
\begin{split}
\mc{F}_R(\exp \gamma) &= F^\perp_{\exp(\gamma)^{\ast}A}+[\exp(-\gamma)\Phi\exp(\gamma)\wedge(\exp(-\gamma)\Phi\exp(\gamma))^\ast] \\
&=\pr_1(\mathcal{H}_R(A,\Phi))+ L_R \gamma + Q_R(\gamma),
\end{split}
\end{array}
\end{equation}
where we set
\begin{align*}
Q_R(\gamma):=& d_{A}(R_{A}(\gamma))+[R_{\Phi}(\gamma)\wedge\Phi^{\ast}]+ [\Phi\wedge R_{\Phi}(\gamma)^{\ast}]\\ 
&+\frac 12 [((\bar\partial_{A}-\partial_{A})\gamma+R_{A}(\gamma))\wedge ((\bar\partial_{A}-\partial_{A})\gamma + R_{A}(\gamma))]\\
&+[([\Phi\wedge \gamma]+R_{\Phi}(\gamma))\wedge([\Phi\wedge \gamma]+R_{\Phi}(\gamma))^{\ast}].
\end{align*}

\begin{lemma}
In the above, let $(A,\Phi)=(A_R^\app, \Phi_R^\app)$. Then there exists a constant $C>0$ such that    
\begin{equation}\label{eq:Lipschitzhigherorder}
\|Q_R(\gamma_1)-Q_R(\gamma_0)\|_{L^2(\Sigma_t)}\leq C r \|\gamma_1-\gamma_0\|_{H_B^2(\Sigma_t)}
\end{equation}
for all $0<r \leq1$ and  $\gamma_0,\gamma_1\in B_{r}$, the closed ball of radius $r$ around $0$ in $H_B^2(\Sigma_t)$.
\end{lemma}

\begin{proof}
The proof follows closely the lines of \cite[Lemma 6.8]{msww14} and has   two steps. For ease of notation, we write $(A,\Phi)$ for $(A_R^\app, \Phi_R^\app)$. Along the proof we shall at several places make use of the continuous embeddings
\begin{align*}
H_B^2(\Sigma_t)\hookrightarrow C^0(\Sigma_t),\qquad H_B^2(\Sigma_t)\cdot H_B^1(\Sigma_t)\hookrightarrow H_B^1(\Sigma_t),\\
 H_B^1(\Sigma_t)\cdot H_B^1(\Sigma_t)\hookrightarrow L^2(\Sigma_t). 
\end{align*}
The corresponding estimates for the norms hold with constants uniform in $R$. This follows since the restriction of the connection $B$ to each euclidean cylinder $\mathcal C_p(R)$ is invariant under  translations along the $\tau$-direction.  
\setcounter{step}{0}
\begin{step}\label{step:remain.terms}
We first check that if $r \in (0,1]$ and $\gamma_0, \gamma_1\in B_{r}$, then 
\begin{align*}
&\|R_A(\gamma_1)-R_A(\gamma_0) \|_{H_B^1(\Sigma_t)}\leq Cr\|\gamma_1 -\gamma_0\|_{H_B^2(\Sigma_t)},\\
&\|R_\Phi(\gamma_1) - R_\Phi(\gamma_0) \|_{H_B^1(\Sigma_t)} \leq Cr\|\gamma_1 -\gamma_0\|_{H_B^2(\Sigma_t)}. 
\end{align*}
To show the first inequality we begin by estimating the difference of the  terms involving $\bar\partial_A$ on the right in Eq.\ \eqref{eq:remainderRA}: 
\begin{multline*}
\big\|\exp(-\gamma_1) (\bar\partial_A(\exp \gamma_1))-\exp(-\gamma_0)(\bar\partial_A(\exp \gamma_0))-\bar\partial_A 
(\gamma_1-\gamma_0) \big\|_{H_B^1(\Sigma_t)} \\[0.5em]
\leq \| (\exp(-\gamma_1)-\exp(-\gamma_0))\bar\partial_A(\exp(\gamma_1))\|_{H_B^1(\Sigma_t)} \\[0.5em]
+\|\exp(-\gamma_0)\big(\bar\partial_A(\exp(\gamma_1)-\exp(\gamma_0))\big)-\bar\partial_A(\gamma_1-\gamma_0)\|_{H_B^1(\Sigma_t)} 
=: \mathrm{I} + \mathrm{II}.
\end{multline*}
Writing $\exp(\gamma)=1+\gamma+S(\gamma)$, then we have 
\begin{align*}
\|\mathrm{I}\|_{H_B^1(\Sigma_t)}&\leq C_0\|\exp(-\gamma_1)-\exp(-\gamma_0)\|_{H_B^2}\|\bar\partial_A\exp(\gamma_1)\|_{H_B^1(\Sigma_t)}\\
&\leq C_1 \|\gamma_1-\gamma_0\|_{H_B^2}\|\gamma_1+S(\gamma_1)\|_{H_B^2(\Sigma_t)}\\
&\leq C_2 r\|\gamma_1-\gamma_0\|_{H_B^2(\Sigma_t)},
\end{align*}
where for   $\gamma\in H_B^2(\Sigma_t)$ we used the estimate
\begin{align*}
\|\bar\partial_A\gamma\|_{H_B^1(\Sigma_t)}\leq& \|\bar\partial_B\gamma\|_{H_B^1(\Sigma_t)}+\|[(B^{0,1}-A^{0,1})\wedge\gamma]\|_{H_B^1(\Sigma_t)}\\
\leq& \| \gamma\|_{H_B^1(\Sigma_t)}+\|B-A\|_{H_B^1(\Sigma_t)}\|\gamma\|_{H_B^2(\Sigma_t)}\\
\leq& C\|\gamma\|_{H_B^2(\Sigma_t)}
\end{align*}
for some constant $C$ which can be chosen uniform in $R$. Similarly, 
\begin{align*}
\lefteqn{\|\mathrm{II}\|_{H_B^1(\Sigma_t)}}\\
=&\,\|(1-\gamma_0+S(-\gamma_0))\big(\bar\partial_A(\gamma_1-\gamma_0+S(\gamma_1)-S(\gamma_0)\big)-\bar\partial_A(\gamma_1-\gamma_0)\|_{H_B^1(\Sigma_t)}\\
\leq&\,\|\bar\partial_A(S(\gamma_1)-S(\gamma_0))\|_{H_B^1(\Sigma_t)}\\
&+\|(-\gamma_0+S(-\gamma_0))\bar\partial_A(\gamma_1-\gamma_0+S(\gamma_1)-S(\gamma_0))\|_{H_B^1(\Sigma_t)}\\
\leq&\, C_0 \|S(\gamma_1)-S(\gamma_0)\|_{H_B^2(\Sigma_t)}\\
&+C_0 \| -\gamma_0+S(-\gamma_0)\|_{H_B^2}\|\gamma_1-\gamma_0+S(\gamma_1)-S(\gamma_0)\|_{H_B^2(\Sigma_t)}\\
\leq&\,C_1 r\|\gamma_1-\gamma_0\|_{H_B^2(\Sigma_t)},
\end{align*}
since
\begin{equation*}
\|S(\gamma_1)-S(\gamma_0)\|_{H_B^2(\Sigma_t)}\leq\|\gamma_1-\gamma_0\|_{H_B^2(\Sigma_t)}\sum_{k\geq1}r^k/k! 
\leq Cr\|\gamma_0-\gamma_1\|_{H_B^2(\Sigma_t)}. 
\end{equation*}
These estimates together with analogous ones for the terms involving 
$\partial_{A}$ give the stated Lipschitz estimate for $R_{A}$. The corresponding estimate for 
\begin{equation*}
R_\Phi= \exp(-\gamma)\Phi\exp\gamma-[\Phi\wedge \gamma]-\Phi 
\end{equation*}
and the estimates
\begin{equation*}
\|R_A(\gamma)\|_{H_B^1(\Sigma_t)}\leq C r,\qquad\|R_\Phi(\gamma)\|_{H_B^1(\Sigma_t)}\leq Cr	
\end{equation*}
for $\gamma\in B_r$ follow in the same way. 
\end{step}

\begin{step}
We  prove the claim. First,
\begin{equation*}\label{eq:LipschitzQ}
\begin{aligned}
Q_R(\gamma_1) - & Q_R(\gamma_0) = d_{A} (R_{A}(\gamma_1) - R_{A}(\gamma_0))\\[0.5ex]
&+  [ (R_{\Phi}(\gamma_1) - R_{\Phi}(\gamma_0)) \wedge \Phi^*] +  [\Phi \wedge (R_{\Phi}(\gamma_1) - R_{\Phi}(\gamma_0))^*]\\[0.5ex]
&+ \tfrac 12 [((\bar\partial_{A}-\partial_{A})\gamma_1 + R_{A}(\gamma_1)) \wedge ((\bar\partial_{A}-\partial_{A})\gamma_1 + 
R_{A}(\gamma_1))]\\[0.5ex]
&- \tfrac 12 [((\bar\partial_{A}-\partial_{A})\gamma_0 + R_{A}(\gamma_0)) \wedge ((\bar\partial_{A}-\partial_{A})\gamma_0 + 
R_{A}(\gamma_0))]\\[0.5ex]
&+ [([\Phi \wedge \gamma_1] + R_{\Phi}(\gamma_1)) \wedge ([\Phi \wedge \gamma_1] + R_{\Phi}(\gamma_1))^* ]\\[0.5ex]
&- [([\Phi \wedge \gamma_0] + R_{\Phi}(\gamma_0)) \wedge ([\Phi \wedge \gamma_0] + R_{\Phi}(\gamma_0))^* ]. 
\end{aligned}
\end{equation*}
Using that $\| B-A\|_{C^0(\Sigma_t)}\leq C$ for some constant $C$ independent of $R$ we obtain
\begin{equation*}
\|d_{A } (R_{A }(\gamma_1) - R_{A }(\gamma_0))\|_{L^2(\Sigma_t)}\leq C \|  R_{A}(\gamma_1) - R_{A}(\gamma_0)\|_{H_B^1(\Sigma_t)},	
\end{equation*}
and we then apply Step~\ref{step:remain.terms}. The remaining terms are bilinear combinations $B(\lambda,\psi)$ of functions 
$\lambda$ and $\psi$ with fixed coefficients, which can be estimated as 
\begin{align*}
\lefteqn{\|B(\lambda_1,\psi_1)-B(\lambda_0,\psi_0)\|_{L^2(\Sigma_t)}}\\
\leq&\|B(\lambda_1,\psi_1-\psi_0)\|_{L^2(\Sigma_t)}+\|B(\lambda_1-\lambda_0,\psi_0)\|_{L^2(\Sigma_t)}\\
\leq& C\|\psi_1-\psi_0\|_{H_B^1(\Sigma_t)}\|\lambda_1\|_{H_B^1(\Sigma_t)} +C\|\psi_0\|_{H_B^1(\Sigma_t)}\|\lambda_1-\lambda_0\|_{H_B^1(\Sigma_t)}.
\end{align*}
The desired estimate follows from Step~\ref{step:remain.terms} again. 
\end{step}
This completes the proof of the lemma.
\end{proof}

\subsection{Contraction mapping argument}\label{subsect:contractionmapping}
%

After these preparations we are in position to show that for every $0<R<R_0$ the approximate solution  $(A_R^{\app},\Phi_R^{\app})$ may be perturbed to  a nearby exact solution, thus proving our main theorem.

\begin{theorem}\label{thm:contrthm}
There exist a constant $0<R_0<1$ and for every $0<R<R_0$ a constant $\sigma_R>0$ and a unique section $\gamma\in H_B^2(\Sigma_t,i\mathfrak{su}(E))$ satisfying $\|\gamma\|_{H_B^2(\Sigma_t)}\leq \sigma_R$ with the following significance.  Set $g=\exp(\gamma)$. Then
\begin{equation*}
(A_R,\Phi_R)=g^{\ast}(A_R^{\app},\Phi_R^{\app})
\end{equation*}
is a solution of Eq.\  \eqref{eq:hitequ} on the smooth surface $\Sigma_t$. 
\end{theorem}

\begin{proof}
Our argument rests on the Banach fixed point theorem. By Eq.\ \eqref{eq:TaylorexpF} there is the expansion  
\begin{equation*}
\mathcal F_R(\exp\gamma)=\pr_1(\mathcal H_R(A_R^{\app},\Phi_R^{\app}))+L_R\gamma+Q_R(\gamma),
\end{equation*}
where the term $Q_R$ involves quadratic and higher order terms in $\gamma$. Then $ \mathcal F_R(\exp\gamma)=0$  if and only if $\gamma$ is a fixed point of the map
\begin{equation*}
T\colon \gamma\mapsto -G_R\left(\pr_1(\mathcal H_R(A_R^{\app},\Phi_R^{\app}))+Q_R(\gamma)\right),
\end{equation*}
where we denote as above $G_R=L_R^{-1}$. We show that $T$ has a unique fixed point as a map $T\colon B_{\sigma_R} \to B_{\sigma_R} $ with $B_{\sigma_R} $ denoting the open ball of radius $\sigma_R$ in $H_B^2(\Sigma_t)$. We claim that for $\sigma_R>0$ sufficiently small, $T$ is a contraction of $ B_{\sigma_R}$, from which we immediately obtain a unique fixed point $\gamma\in B_{\sigma_R}$. To show this, we use Corollary \ref{lem:estL2H2norm} and the inequality \eqref{eq:Lipschitzhigherorder} to obtain
\begin{eqnarray*}
\|T(\gamma_1-\gamma_0)\|_{H_B^2(\Sigma_t)} &=&\|G_R(Q_R(\gamma_1)-Q_R(\gamma_0))\|_{H_B^2(\Sigma_t)}\\
&\leq&C(\log R)^2\|Q_R(\gamma_1)-Q_R(\gamma_0)\|_{L^2(\Sigma_t)}\\
&\leq&C(\log R)^2\sigma_R    \|\gamma_1-\gamma_0\|_{H_B^2(\Sigma_t)}.
\end{eqnarray*}
Let $\epsilon>0$ and set $\sigma_R:=C^{-1}\left|\log R\right|^{-2-\epsilon}$. Then for all   $0<R<e^{-1}$ it follows that  $C(\log R)^2\sigma_R<1$ and therefore $T$ is a contraction on the ball of radius $\sigma_R$. Furthermore, since $Q_R(0)=0$, 
using  again  Corollary \ref{lem:estL2H2norm}  and Lemma \ref{lem:errorappr} it follows that  
\begin{eqnarray*}
\|T(0)\|_{H_B^2(\Sigma_t)}&=&  \|G_t(\pr_1(\mathcal H_R(A_R^{\app},\Phi_R^{\app})))\|_{H_B^2(\Sigma_t)}\\
&\leq&C(\log R)^2\|\pr_1(\mathcal H_R(A_R^{\app},\Phi_R^{\app}))\|_{L^2(\Sigma_t)}\\
&\leq&C(\log R)^2R^{\delta''}.
\end{eqnarray*}
Thus when $R_0$ is chosen to be sufficiently small,  then  $\|T(0)\|_{H_B^2(\Sigma_t)} < \frac{1}{10}\sigma_R$ for all $0<R<R_0$ and  the above choice of $\sigma_R$, so the ball $B_{\sigma_R}$ is mapped to itself by $T$. Therefore the Banach fixed point theorem applies and yields the existence of a unique solution $\gamma$ as desired.
\end{proof}

The proof of the main theorem is now an almost immediate consequence of Theorem \ref{thm:contrthm}.

\begin{proof}{\bf{[Theorem \ref{thm:mainthm}].}}
We claim  that the family of solutions $(A_R,\Phi_R)$  obtained in Theorem \ref{thm:contrthm} satisfies the asserted properties. Indeed, for any fixed compact subset $K\subseteq\Sigma_0$ and all sufficiently small $0<R<R_0$ it follows from the above that
\begin{multline*}
\|(A_R,\Phi_R)-(A_R^{\app},\Phi_R^{\app})\|_{L^2(K)}\\
=\|\exp(\gamma)^{\ast}(A_R^{\app},\Phi_R^{\app})-(A_R^{\app},\Phi_R^{\app})\|_{L^2(K)}
\leq C\|\gamma\|_{H_B^2(K)}\leq C\sigma_R
\end{multline*}
for some uniform (in the parameter  $R$) constant $C$, and with $\lim_{R\to0}\sigma_R=0$. By a standard elliptic bootstrapping argument, applied to Eq.\ \eqref{eq:TaylorexpF} with $\mathcal F_R(\exp\gamma)=0$, it in fact holds that $\gamma\to 0$ as $R\to0$ in every norm $H_B^k(K)$, $k\geq2$. Since $H_B^k(K)$ embeds into $C^{\ell}(K)$ for all  sufficiently large $k$, and because $(A_R^{\app},\Phi_R^{\app})\to (A_0^{\app},\Phi_0^{\app})$ uniformly on $K$, we obtain the uniform convergence of $(A_R,\Phi_R)$ to $(A_0^{\app},\Phi_0^{\app})$ as $R\to0$. Since by construction the approximate solution $(A_0^{\app},\Phi_0^{\app})$ coincides   with the given exact one, the claim follows. 
\end{proof}

%
%
%

\end{document}